\documentclass[11pt,onecolumn]{article}

\usepackage{microtype}
\usepackage{graphicx}
\usepackage{subfigure}
\usepackage{booktabs} % for professional tables

\usepackage{amsmath,amsthm,amssymb,amsfonts,epsfig,graphicx}
\usepackage{color}
\usepackage{dsfont}
\usepackage{wrapfig}

\usepackage{enumitem}
\usepackage[vlined,linesnumbered,ruled,resetcount]{algorithm2e}

\usepackage[left=1.0in,top=1.0in,right=1.05in,bottom=0.5in,nohead,textheight=10in,footskip=0.3in]{geometry}

\newtheorem{theorem}{Theorem}
\newtheorem{corollary}[theorem]{Corollary}

\newtheorem{lemma}[theorem]{Lemma}

\newtheorem{proposition}[theorem]{Proposition}

\newtheorem{remark}{Remark}

\DeclareMathOperator*{\argmin}{arg\,min}
\DeclareMathOperator*{\argmax}{arg\,max}
\DeclareMathOperator*{\dom}{dom}

\SetKwInput{KwInput}{Input}                % Set the Input
\SetKwInput{KwOutput}{Output}              % set the Output

\begin{document}

\def\myparagraph#1{\vspace{2pt}\noindent{\bf #1~~}}

%\pagestyle{headings}

%%%%%%%%%%%%%%%%%%%%%%%%%%%%%%%%%%%%%%%%%%%%%%%%%%%%%%%%%%%
%%%%%%%%%%%%%%%%%%%%%%%%%%%%%%%%%%%%%%%%%%%%%%%%%%%%%%%%%%%
%%%%%%%%%%%%%%%%%%%%%%%%%%%%%%%%%%%%%%%%%%%%%%%%%%%%%%%%%%%

\newcommand{\eqdef}{{\stackrel{\mbox{\tiny \tt ~def~}}{=}}}

\def\calQ{{\cal Q}}
\def\calW{{\cal W}}
\def\A{{\mathbb A}}

\def\DeltaCeil{{\lceil\Delta\rceil}}
\def\TwoDeltaCeil{{\lceil 2\Delta\rceil}}
\def\OnePointFiveDeltaCeil{{\lceil 3\Delta/2\rceil}}

\long\def\ignore#1{}
%\epsfverbosetrue
%\def\myps[#1]#2{}
\def\myps[#1]#2{\includegraphics[#1]{#2}}
\def\etal{{\em et al.}}
\def\Bar#1{{\bar #1}}
\def\br(#1,#2){{\langle #1,#2 \rangle}}
\def\setZ[#1,#2]{{[ #1 .. #2 ]}}
\def\Pr{\mbox{\rm Pr}}
\def\REACHED{\mbox{\tt REACHED}}
\def\AdjustFlow{\mbox{\tt AdjustFlow}}
\def\GetNeighbors{\mbox{\tt GetNeighbors}}
\def\true{\mbox{\tt true}}
\def\false{\mbox{\tt false}}
\def\Process{\mbox{\tt Process}}
\def\ProcessLeft{\mbox{\tt ProcessLeft}}
\def\ProcessRight{\mbox{\tt ProcessRight}}
\def\Add{\mbox{\tt Add}}

\def\setof#1{{\left\{#1\right\}}}
\def\suchthat#1#2{\setof{\,#1\mid#2\,}} % so says Knuth, page 174
\def\event#1{\setof{#1}}
\def\q={\quad=\quad}
\def\qq={\qquad=\qquad}
\def\calA{{\cal A}}
\def\calB{{\cal B}}
\def\calC{{\cal C}}
\def\calD{{\cal D}}
\def\calE{{\cal E}}
\def\calF{{F}}
\def\calG{{\cal G}}
\def\calI{{\cal I}}
\def\calJ{{\cal J}}
\def\calH{{\cal H}}
\def\calL{{\cal L}}
\def\calN{{\cal N}}
\def\calP{{\cal P}}
\def\calR{{\cal R}}
\def\calS{{\cal S}}
\def\calT{{\cal T}}
\def\calU{{\cal U}}
\def\calV{{\cal V}}
\def\calO{{\cal O}}
\def\calX{{\cal X}}
\def\calY{{\cal Y}}
\def\s{\footnotesize}
\def\calNG{{\cal N_G}}
\def\psfile[#1]#2{}
\def\psfilehere[#1]#2{}
\def\epsfw#1#2{\includegraphics[width=#1\hsize]{#2}}
\def\assign(#1,#2){\langle#1,#2\rangle}
\def\edge(#1,#2){(#1,#2)}
\def\VS{\calV^s}
\def\VT{\calV^t}
\def\slack(#1){\texttt{slack}({#1})}
\def\barslack(#1){\overline{\texttt{slack}}({#1})}
\def\NULL{\texttt{NULL}}
\def\PARENT{\texttt{PARENT}}
\def\GRANDPARENT{\texttt{GRANDPARENT}}
\def\TAIL{\texttt{TAIL}}
\def\HEADORIG{\texttt{HEAD$\_\:$ORIG}}
\def\TAILORIG{\texttt{TAIL$\_\:$ORIG}}
\def\HEAD{\texttt{HEAD}}
\def\CURRENTEDGE{\texttt{CURRENT$\!\_\:$EDGE}}

\def\unitvec(#1){{{\bf u}_{#1}}}
\def\uvec{{\bf u}}
\def\vvec{{\bf v}}
\def\Nvec{{\bf N}}

\newcommand{\bg}{\mbox{$\bf g$}}
\newcommand{\bh}{\mbox{$\bf h$}}

\newcommand{\bx}{\mbox{$x$}}
\newcommand{\by}{\mbox{\boldmath $y$}}
\newcommand{\bz}{\mbox{\boldmath $z$}}
\newcommand{\bu}{\mbox{\boldmath $u$}}
\newcommand{\bv}{\mbox{\boldmath $v$}}
\newcommand{\bw}{\mbox{\boldmath $w$}}
\newcommand{\bvarphi}{\mbox{\boldmath $\varphi$}}

\newcommand\myqed{{}}

\newcommand{\scp}[3][]{\langle{#2},\, {#3}\rangle_{#1}}
\newcommand{\prox}{{\tt prox}}
\newcommand{\proj}{{\tt proj}}
\newcommand{\lmo}{{\tt lmo}}
\newcommand{\norm}[2][]{\|{#2}\|_{#1}}

\graphicspath{{./code/figures/}}

%%%%%%%%%%%%%%%%%%%%%%%%%%%%%%%%%%%%%%%%%%%%%%%%%%%%%%%%%%%
%%%%%%%%%%%%%%%%%%%%%%%%%%%%%%%%%%%%%%%%%%%%%%%%%%%%%%%%%%%
%%%%%%%%%%%%%%%%%%%%%%%%%%%%%%%%%%%%%%%%%%%%%%%%%%%%%%%%%%%

\title{\Large\bf  \vspace{0pt} One-sided Frank-Wolfe algorithms for saddle problems}
\author{Vladimir Kolmogorov

  \footnote{Vladimir Kolmogorov was supported by the European Research Council
    under the European Unions Seventh Framework Programme
    (FP7/2007-2013)/ERC grant agreement no 616160.}

  \\ \normalsize Institute of Science and Technology Austria
  \\ {\normalsize\tt vnk@ist.ac.at}
  \and
  Thomas Pock
  \footnote{Thomas Pock acknowledges support by an ERC grant HOMOVIS,
    no 640156.}  \\ \normalsize Institute of Computer Graphics and
  Vision, \\Graz University of Technology\\
  {\normalsize\tt pock@icg.tugraz.at}}

\date{}
\maketitle

%\vspace{-6pt}
\begin{abstract}
  We study a class of convex-concave saddle-point problems of the form
  $\min_x\max_y \langle Kx,y\rangle+f_\calP(x)-h^\ast(y)$ where $K$ is
  a linear operator, $f_\calP$ is the sum of a convex function $f$
  with a Lipschitz-continuous gradient and the indicator function of a
  bounded convex polytope $\calP$, and $h^\ast$ is a convex (possibly
  nonsmooth) function. Such problem arises, for example, as a
  Lagrangian relaxation of various discrete optimization problems. Our
  main assumptions are the existence of an efficient {\em linear
    minimization oracle} ($\lmo$) for $f_\calP$ and an efficient {\em
    proximal map} ($\prox$) for $h^*$ which motivate the solution via
  a blend of proximal primal-dual algorithms and Frank-Wolfe
  algorithms. In case $h^*$ is the indicator function of a linear
  constraint and function $f$ is quadratic, we show a $O(1/n^2)$
  convergence rate on the dual objective, requiring $O(n \log n)$
  calls of $\lmo$. If the problem comes from the constrained
  optimization problem $\min_{x\in\mathbb
    R^d}\{f_\calP(x)\:|\:Ax-b=0\}$ then we additionally get bound
  $O(1/n^2)$ both on the primal gap and on the infeasibility gap.  In
  the most general case, we show a $O(1/n)$ convergence rate of the
  primal-dual gap again requiring $O(n\log n)$ calls of $\lmo$. To the
  best of our knowledge, this improves on the known convergence rates
  for the considered class of saddle-point problems.  We show
  applications to labeling problems frequently appearing in machine
  learning and computer vision.
\end{abstract}

\section{Introduction}\label{sec:intro}

In this paper, we consider the following class of saddle-point problems:
\begin{equation}
\min_{x \in \calX} \max_{y\in \calY} \calL(x,y) := \scp{Kx}{y} + f_\calP(x) - h^*(y)
\label{eq:saddle}
\end{equation}
where $\calX,\calY$ are finite dimensional spaces, equipped with an
inner product $\scp{\cdot}{\cdot}$ and $K: \calX \to \calY$ is a
bounded linear operator with operator norm $L_K = \norm{K}$. Usually
the underlying spaces are the standard Euclidean spaces $\calX =
\mathbb R^n$ and $\calY=\mathbb R^m$.

The functions $f_\calP(x)$ and $h^*$ are convex, lower semicontinuous
functions.

For a differentiable convex function $f$ we say that it has a
Lipschitz continuous gradient if there exists a constant $L_f \geq 0$
such that
\[
\norm{\nabla f(x) - \nabla f(y)} \leq L_f \norm{x-y}, \quad \forall
x,y \in \calX.
\]
Moreover, the function $f$ is called strongly convex with strong convexity
parameter $\mu_f > 0$ if 
\[
f(y) \geq f(x) + \scp{\nabla f(x)}{y-x} + \frac{\mu_f}{2}
\norm{y-x}^2, \quad \forall x,y \in \calX.
\]

We make the following important structural assumptions on the
functions $f_\calP(x)$ and $h^*$:
\begin{itemize}
\item The function $f_\calP(x)$ has the following composite form
  \[
  f_\calP(x) = f(x) + \delta_{\calP}(x),
  \]
  where $f$ is a convex function with a $L_f$-Lipschitz continuous
  gradient and $\delta_{\calP}$ is the indicator function of a convex
  polytope $\calP \subset \calX$. For this polytope, we assume the
  existence of an efficient \textbf{linear minimization oracle} ($\lmo$), which means
  that for any $a \in \calX^*$, one can efficiently solve
  \[
  \lmo_{\calP}(a) \in \arg\min_{x\in \calP} \scp{a}{x}.
  \]
  This is for example the case if $\calP$ is the polytope arising from
  LP relaxations of MAP-MRF problems in a tree-structured graph, where the above problem can
  be solved efficiently using dynamic programming.
\item The function $h^*$ is a convex function which allows to
  efficiently compute its \textbf{proximal map} ($\prox$), which for
  any $\bar y \in \calY$ and $\tau > 0$ is defined as
  \[
  \prox_{\tau h^*}(\bar y) = \arg\min_{y\in \calY} \frac1{2\tau}\norm{y-\bar y}^2 + h^*(y).
  \]
  Important examples of $h^*$ which allow for an efficient proximal
  map include quadratic functions and various norms. If $h^*=\delta_C$
  i.e. the indicator functions of some convex set $C$ the proximal map
  reduces the orthogonal projection operator.
\end{itemize}

%Some of our results will be specialized to saddle point problems corresponding
%to the problem of minimizing $f_\calP(x)$ subject to the linear constraint $Ax-b=0$.

%Some of our results will be specialized to problems of the form
An important special case of problem~\eqref{eq:saddle} is given by 
\begin{equation}
\min_{x \in \mathbb R^d} \max_{y\in \mathbb R^d} \calL(x,y) := f_\calP(x) + \scp{y}{Ax-b}
\label{eq:saddle:Axb}
\end{equation}
where $A$ is a matrix and $b$ is a vector of appropriate dimensions.
This corresponds to the problem of minimizing $f_\calP(x)$ subject to the linear constraint $Ax-b=0$.

\myparagraph{Primal and dual problems}
We denote by $(x^\star,y^\star)$ a saddle point of
problem~\eqref{eq:saddle}, which satisfies
\[
\calL(x^\star,y) \leq \calL(x^\star,y^\star) \leq \calL(x,y^\star),
\quad \forall (x,y) \in \calX \times \calY.
\]
Throughout the paper we denote the primal and dual problems respectively by
\[
\calF(x) = \max_{y\in\calY} \calL(x,y), \quad H(y) = \min_{x\in\calX} \calL(x,y).
\]
We assume that strong duality holds, that is
\[
H(y^\star)=\max_{y\in\calY} H(y)=\max_{y\in\calY}\min_{x\in\calX}
    \calL(x,y)
    =\min_{x\in\calX} \calF(x)=\calF(x^\star)
\]
Some of the results will also assume coercivity of a function. We
recall that a proper function $\phi(z)$ is called coercive if
\[
\lim_{\norm{z} \to \infty} \phi(z) = \infty.
\]

\myparagraph{Contributions}
The algorithms we propose here are based on inexact proximal
algorithms, which allow for an approximate evaluation of the proximal
maps. For this we make use of efficient variants of the Frank-Wolfe
algorithm that offer a linear convergence rate on the proximal
subproblems. In summary, after $O(n \log n)$ calls to $\lmo$ we achieve the following guarantees.
\begin{itemize}
\item If function $f$ is linear or quadratic, $h^*$ is the indicator function of a linear
  constraint, and function $-H(y)$ is coercive then
  we obtain accuracy $O(1/n^2)$  on the dual objective $H$. % (and also on the primal objective $F$ assuming that $\calY$ is compact).
If in addition the problem has the form of eq.~\eqref{eq:saddle:Axb} 
then we obtain accuracy $O(1/n^2)$  both on the primal gap $f_\calP(x)-f_\calP(x^\star)$ and
 on the infeasibility gap $||Ax-b||$.

%To our knowledge, the best prior work (described later in Sec.~\ref{sec:related}) after $n$ calls to $\lmo$ obtained  
%accuracy $O(1/n)$ on the dual objective and accuracy $O(1/\sqrt{n})$ on both $f_\calP(x)-f_\calP(x^\star)$ and $||Ax-b||$.
\item In the most general case,  we obtain accuracy $O(1/n)$  on the dual objective (and also on the primal objective assuming that $\dom h^\ast\subseteq\calY$ is compact).
\end{itemize}

To the
best of our knowledge these rates improve on the so far known rates
for the class of saddle-point problems considered in this paper.
In particular, for the problem in eq.~\eqref{eq:saddle:Axb} previous works 
(described later in Sec.~\ref{sec:related}) after $n$ calls to $\lmo$ obtained  
accuracy $O(1/n)$ on the dual objective and accuracy $O(1/\sqrt{n})$ on both $f_\calP(x)-f_\calP(x^\star)$ and $||Ax-b||$.

\subsection{Motivating example}\label{sec:example}

An important application, which also serves as the main motivation for
the class of saddle-point problems studied in this paper, is given by
the Lagrangian relaxation of discrete optimization problems.
To form such relaxation,
one needs to first 
%It can be described as follows.
%First, let us 
encode discrete variables via Boolean indicator variables $X\in\{0,1\}^d$,
and then express a difficult optimization problem as a sum of tractable subproblems:
\begin{equation}\label{eq:discrete}
\min_{X\in\{0,1\}^d} 
%f(X) := 
\;\;\sum_{t\in T}f_t(X_{A_t})
\end{equation}
Here $T=\{1,\ldots,m\}$ is the set of terms where each term $t$ is
specified by a subset of variables $A_t\subseteq[d]$
and a function $f_t:\{0,1\}^{A_t}\rightarrow\mathbb R\cup\{+\infty\}$ of $|A_t|$ variables.
Vector $X_{A_t}$ is the restriction of vector $X\in\mathbb R^d$ to $A_t$.
The arity $|A_t|$ of function $f_t$ can be arbitrarily large, however 
we assume the existence of an efficient {\em min-oracle} that for a given vector $Y\in\mathbb R^{A_t}$ computes 
%$X\in\argmin\limits_{X\in \dom f_t} \left[f_t(X)+\langle X,Y\rangle\right]$
%together with the cost $f_t(X)$,
%where  $\dom f_t=\{X\in\{0,1\}^{A_t}\:|\:f_t(X)<+\infty\}\ne\varnothing$ is the effective domain of $f_t$.
$X\in\argmin\nolimits_{X\in \{0,1\}^{A_t}} \left[f_t(X)+\langle X,Y\rangle\right]$
together with the cost $f_t(X)$. 
For example, this holds if $f_t(\cdot)$ corresponds to a MAP-MRF inference problem in a tree-structured graph.

It is now easy to define a relaxation of~\eqref{eq:discrete} that will have the form of eq.~\eqref{eq:saddle} (see e.g.~\cite{Swoboda:CVPR19}).
%In order to define a relaxation of problem~\eqref{eq:discrete},
First, for each $t\in T$ let us define polytope $\calP^t={\tt conv}(\{[X\;f(X)]\:|\:X\in\dom f_t\})\subseteq \mathbb R^{A_t}\times R$
where $\dom f_t=\{X\in\{0,1\}^{A_t}\:|\:f_t(X)<+\infty\}$ is the effective domain of $f_t$.
Problem~\eqref{eq:discrete} can then be equivalently written as
\begin{equation}\label{eq:discrete2}
\min_{X\in\{0,1\}^d,X^1\in\calP^{1},\ldots,X^m\in\calP^{m}} \;\;\sum_{t\in T} X^t_\circ\qquad\quad \mbox{s.t.}\quad X_v=X^t_v\qquad \forall t\in T,v\in A_t
\end{equation}
%over variables $x=(X,X^1,\ldots,X^m)$ 
where $X^t_\circ$ denotes the last component of vector $X^t\in \mathbb R^{A_t}\times R$.
%
%
%
%By dropping constraint $X\in\{0,1\}^d$ and dualizing equalities $X_v=X^t_v$
%we obtain the following saddle point problem, which a relaxation of~\eqref{eq:discrete2} (its optimal value is a lower bound on~\eqref{eq:discrete2}):
%\begin{equation}\label{eq:discrete3}
%\min_{X\in\mathbb R^d,X^1\in\calP^{1},\ldots,X^m\in\calP^{m}} \;\;\;\max_{y}\;\;\sum_{t\in T} X^t_\circ+\sum_{t\in T,v\in A_t} Y^t_v (X_v^t-X^t) 
%\end{equation}
%where $Y^t_v$ is the Lagrange multiplier for the constraint $X^t_v-X_v=0$, and we denoted $y=(Y^t_v)_{t\in T,v\in A_t}$. 
%Variables $X$ can be easily eliminated, yielding the following equivalent problem:
%\begin{equation}\label{eq:discrete3}
%\min_{X^1\in\calP^{1},\ldots,X^m\in\calP^{m}} \;\;\;\max_{y\in\calY}\;\;\sum_{t\in T} X^t_\circ+\sum_{t\in T,v\in A_t} X^t_v Y^t_v 
%\end{equation}
%
%
By dropping constraint $X\in\{0,1\}^d$, dualizing equalities $X^t_v-X_v=0$ with a Lagrange multiplier $Y_v^t$, and
eliminating variables $X$,
we obtain the following saddle point problem which is a relaxation of~\eqref{eq:discrete2} (its optimal value is a lower bound on~\eqref{eq:discrete2}):
%\begin{equation}\label{eq:discrete3}
%\min_{X\in\mathbb R^d,X^1\in\calP^{1},\ldots,X^m\in\calP^{m}} \;\;\;\max_{y}\;\;\sum_{t\in T} X^t_\circ+\sum_{t\in T,v\in A_t} Y^t_v (X_v^t-X^t) 
%\end{equation}
%Variables $X$ can be easily eliminated, yielding the following equivalent problem:
\begin{equation}\label{eq:discrete3}
\min_{x=(X^1,\ldots,X^m)\in \calP^{1}\times\ldots\times\calP^{m}} \;\;\;\max_{y=(Y^t_v)_{t\in T,v\in A_t}\in\calY}\;\;\sum_{t\in T} X^t_\circ+\sum_{t\in T,v\in A_t} X^t_v Y^t_v 
\end{equation}
%where  we denoted $y=(Y^t_v)_{t\in T,v\in A_t}$. 
%
%
%
where $\calY$ is the set of vectors $y=(Y^t_v)_{t\in T,v\in A_t}$ satisfying constraints $\sum_{t:v\in A_t}Y^t_v=0$ for all $v\in[d]$.
This problem has the form of eq.~\eqref{eq:saddle} 
where function $f(x)$ is linear and $h^\ast(y)$ is the indicator function of a linear constraint on $y$.

As an example, the MAP-MRF inference problem on an undirected graph
 can be cast in the framework above by decomposing the graph into tree-structured subproblems.
It is well-known that the Lagrangian relaxation is equivalent to a standard
LP relaxation, aka the local polytope
relaxation~\cite{DualDecompositionKomodakis,Savchynskyy:book}.
MAP-MRF problems find numerous applications in machine learning and computer
vision~\cite{blake2011markov}.
In more recent
work, they also appears as the final inference layer in deep
convolutional neural networks~\cite{knobelreiter2020belief}.

\subsection{Related work}\label{sec:related}
Saddle point problems in the form of~\eqref{eq:saddle} can be solved
by a large number of proximal primal-dual algorithms (see for example
the recent work~\cite{condat2019proximal} for a very comprehensive
overview) as soon as the proximal maps for both the primal and dual
functions can be solved efficiently. On the other hand, Gidel
\etal~proposed in~\cite{gidel2017frank} an extension of the
Frank-Wolfe algorithm to saddle-point problems
$\min_{x\in\calX}\max_{y\in\calY}\calL(x,y)$ by assuming the existence
of an efficient linear minimization oracle for the product space
$\calX\times\calY$ (which is assumed to be a compact set). In this
paper, we are assuming the existence of an efficient linear
minimization oracle just on the primal and an efficient proximal map
on the dual. Therefore, our algorithms somewhat stand between the two
aforementioned techniques.

Our proposed algorithms rely on the inexact accelerated proximal
gradient algorithm of Aujol and Dossal~\cite{AujolDossal:15} and the
inexact primal-dual algorithm of Rasch and Chambolle~\cite[Section
  3.1]{RaschChambolle:20}. Note that the first method only generates a
dual sequence $\{y_n\}$.  We extend the method and the analysis to
also generate a primal sequence $\{x_n\}$, which is needed to solve
the saddle problem~\eqref{eq:saddle}.

Several authors studied a special case of~\eqref{eq:saddle} given
in~\eqref{eq:saddle:Axb}, or equivalently the problem of minimizing
function $f_\calP(x)=f(x)+\delta_\calP(x)$ subject to linear
constraints $Ax=b$~\cite{Gidel:AISTATS18,Liu:19,Yurtsever:ICML18}.
(The last paper actually considered a more general class of saddle
problems). Papers~\cite{Liu:19,Yurtsever:ICML18} achieve an accuracy
of $O(n^{-1/2})$ after $n$ iterations on the primal and infeasibility
gaps, where \cite{Yurtsever:ICML18} uses one $\lmo$ call per iteration
and~\cite{Liu:19} uses $O(k^2)$ $\lmo$ calls at $k$-th iteration
assuming that a standard Frank-Wolfe method is employed. 
Note, the papers above do not give bounds on suboptimality gaps of the
dual function $H$; instead,~\cite{Gidel:AISTATS18,Liu:19} bound
residuals of the {\em augmented} Lagrangian, which is not directly
related to the residuals of the Lagrangian in
eq.~\eqref{eq:saddle:Axb}.  A similar but slightly more general class
of composite optimization problems was also recently considered
in~\cite{silveti2020}.  In a setting similar to our paper ($\nabla f$
is Lipschitz continuous) they show $O(n^{-1/3})$ accuracy on
Lagrangian values after $n$ calls to $\lmo$.

Frank-Wolfe algorithms for saddle-point problems have also been used in~\cite{Argyriou:14,LanZhou:16}.
The former paper achieved a $O(n^{-1/2})$ convergence rate on a rather general class of constrained optimization problems.
The paper~\cite{LanZhou:16} shares some high-level similarities with our approach
(such as solving smoothed primal subproblem to a given accuracy with Frank-Wolfe),
but uses a different smoothing strategy that requires both primal and dual domains to be compact.
This assumption rules out many interesting applications, including the one considered in Section~\ref{sec:example}.

There is a large body of literature on the special case of problem~\eqref{eq:saddle}
corresponding to Lagrangian relaxation of discrete optimization problems, see e.g.~\cite{
StorvikDahlLagrangeanBasedMAP, % 2000
SchlesingerSubgradient, % 2007
LagrangeanRelaxationJohnsonMalioutov, % 2007
RavikumarProximalMethodsMAPMRF, % JMLR 2010
AcceleratedMAPJojic, % ICML 2010
savchynskyy2011study, % CVPR 2011
schmidt2011evaluationProximalMAP, % EMMCVPR 2011
DualDecompositionKomodakis, % 2011
Martins:ICML11,
AdaptiveDiminishingSmoothingSavchynskyyUAI, % UAI 2012,
MapMirrorDescent, % ISVC 2012
Schwing:NIPS12,
Schwing:ICML14,
Swoboda:CVPR19}. % July 2012
Some of these methods apply only to MAP inference problems in pairwise (or low-order) graphical models,
because they need to compute marginals in tree-structured subproblems~\cite{LagrangeanRelaxationJohnsonMalioutov,AcceleratedMAPJojic,AdaptiveDiminishingSmoothingSavchynskyyUAI}
or because they explicitly exploit the fact that the relaxation can be described by polynomial many constraints~\cite{schmidt2011evaluationProximalMAP,Martins:ICML11,Schwing:NIPS12,Schwing:ICML14}.
%Some of these methods apply only to the MAP inference problems decomposed into trees
%as they need to compute marginal probabilities~\cite{LagrangeanRelaxationJohnsonMalioutov,AcceleratedMAPJojic,AdaptiveDiminishingSmoothingSavchynskyyUAI}, 
%or into edges~\cite{schmidt2011evaluationProximalMAP,Martins:ICML11}.
The papers \cite{AcceleratedMAPJojic,savchynskyy2011study}
obtained accuracy $O(1/n)$ on the dual objective after $n$ iterations, by applying accelerated gradient methods~\cite{Nesterov83}.
% Nesterov's technique to a smoothed objective.

The first method that we develop can be viewed as an extension of the technique in~\cite{Swoboda:CVPR19},
which applied an inexact proximal point algorithm (PPA) to the dual objective.
In contrast to~\cite{Swoboda:CVPR19}, we apply an accelerated version of inexact PPA,
specify to which accuracy the subproblems need to be solved,
and analyze the convergence rate.

%For the particular problem~\eqref{eq:MRF} many algorithms have been
%proposed in the literature including graph cut based
%approaches~\cite{BoykovKolmogorov2004}, belief propagation
%(BP)~\cite{Pearl88}, tree-reweighted message passing
%(TRW)~\cite{wainwright2005map} and its convergent variant, called
%TRW-S~\cite{Kolmogorov2006}. The latter is shown to monotonically
%increase the dual objective $H(\mathbf{y}) = \min_{\mathbf{x}}
%\calL(\mathbf{x},\mathbf{y})$ but might get stuck in a non-optimal
%point. A subgradient ascent method on the dual problem has been
%proposed in~\cite{Komodakis2007}. This guarantees convergence to a
%global optimimum of the dual problem but suffers from a very slow
%$O(1/\sqrt{n})$ convergence rate. In order to obtain faster
%convergence rates, Savchynskyy et
%al.~\cite{Savchynskyy2011,Savchynskyy2012} considered a smooth version
%of the dual problem~\cite{Nesterov05} and adopted accelerated gradient
%methods~\cite{Nesterov83}. Overall, this leads to a $O(1/n)$
%convergence rate. Our proposed algorithm offers a $O(1/n^2)$
%convergence rate requiring $O(n\log n)$ calls of $\lmo$ and hence
%improves on the previous methods.

\subsection{Notation for approximate solutions and organization of the paper}

We introduce the following notation for a function $\phi$ and
accuracy $\varepsilon\ge 0$:
\begin{eqnarray*}
z\approx_\varepsilon\argmin_z\phi(z)\quad&\Leftrightarrow&\quad \phi(z)\le \min_z \phi(z)+\varepsilon \\
z\approx_\varepsilon\argmax_z\phi(z)\quad&\Leftrightarrow&\quad \phi(z)\ge \max_z \phi(z)-\varepsilon \\
z\approx_\varepsilon\prox_{\tau\phi}(\bar z)\quad&\Leftrightarrow&\quad z\approx_\varepsilon\argmin_z \left[\phi(z)+\frac{1}{2\tau}||z-\bar z||^2\right]
\end{eqnarray*}

The rest of the paper is organized as follows. The next section
describes Frank-Wolfe algorithms for minimizing a smooth function over
a convex polytope. Then in Section~\ref{sec:firstalg} we will present our first
approach, which is based on an inexact accelerated proximal point
algorithm on the dual problem. In Section~\ref{sec:secondalg} we present our second
approach, which is based on an inexact proximal primal-dual algorithm
and directly solves the saddle point problem. 
Preliminary numerical results are given in Section~\ref{sec:results}. 
To make the paper more self-contained,
Appendix~\ref{sec:Moreau} discusses some important results of convex
optimization. Additionally, for a better readability, the more
technical proofs are all moved to Appendices~\ref{sec:proof:first}-\ref{sec:proof:last}.

\section{Frank-Wolfe algorithms}\label{sec:FW}
Frank-Wolfe style algorithms is a class of algorithms for minimizing
functions $g_\calP:\calX\rightarrow \mathbb R$ of the form
$g_\calP(x)=g(x)+\delta_\calP(x)$ where $g$ a convex continuously
differentiable function with a Lipschitz continuous gradient and
$\calP$ is a convex polytope.  They are typically iterative techniques
that work by applying a certain procedure ${\tt
  FWstep}(x;g_\calP)\mapsto x'$ where $g_\calP$ is the objective
function, and $x$ and $x'$ are respectively the old and the new
iterates with $g_\calP(x')\le g_\calP(x)$.  We will apply such steps
to functions $g_\calP$ that change from time to time, which is why
$g_\calP$ is made a part of the notation.

It will be convenient to denote
$g^\downarrow_\calP(x)=g_\calP(x)-\min_{x \in\calX}g_\calP(x)$ to be
a shifted version of $g_\calP$ with $\min_{x\in\calX}
g_\calP^\downarrow(x)=0$.  The following fact is known (see~\cite{AFW}).
For completeness, in Appendix~\ref{sec:lemma:FWgap1:proof} we give a proof of the second inequality,
expanding some derivations that were omitted in the proof of~\cite[Theorem 2]{AFW}.
\begin{lemma}[\cite{AFW}]\label{lemma:FWgap1}
For a point $\hat x\in\calP$ denote ${\tt gap}^{\tt FW}(\hat x;g_\calP)=\langle \nabla g(\hat x),\hat x-s\rangle$ where $s=\lmo_\calP(\nabla g(\hat x))$.
Then
$$
g^\downarrow_\calP(\hat x)\le {\tt gap}^{\tt FW}(\hat x;g_\calP)\le
\begin{cases}
g_\calP^\downarrow(\hat x)+LD^2/2 & \mbox{if } g_\calP^\downarrow(\hat x)>LD^2/2 \\
D\sqrt{2L\cdot g_\calP^\downarrow(\hat x)} & \mbox{if } g_\calP^\downarrow(\hat x)\le LD^2/2
\end{cases}
$$
where $D$ is the diameter of $\calP$ and $L=L_g$ is the Lipschitz constant of $\nabla g$.
\end{lemma}

While the original FW algorithm has a sublinear convergence rate, there are several variants that achieve a linear convergence rate
under some assumptions on $g$. 
Examples include {\em Frank-Wolfe with away steps} (AFW)~\cite{AFW},
{\em Decomposition-invariant Conditional Gradient} (DiCG)~\cite{GarberMeshi:16,BashiriZhang:17}, and {\em Blended Conditional Gradient} (BCG)~\cite{BCG}.
Each step in these methods is classified as either {\em good} or {\em bad}.
Good steps are guaranteed to decrease $g_\calP^\downarrow(x)$ by a constant factor.
Bad steps do not have such guarantee (because they hit the boundary of the polytope),
but they make $x$ ``sparser'' in a certain sense and thus cannot happen too often.

More formally, 
consider a class of functions $\mathfrak{F}$ where each function $g_\calP\in\mathfrak{F}$ is associated with a parameter vector $\Theta_{g_\calP}\in\mathbb R^{p}$,
and $\calP=\dom g_\calP$ is the same for all $g_\calP\in\mathfrak{F}$.
We say that procedure ${\tt FWstep}$ has a {\em linear convergence rate on~$\mathfrak{F}$}
if there exist continuous function $\theta:\mathbb R^{p}\rightarrow(0,1)$ and integers $R_0,R_1\ge 0$
with the following properties:
(i) if the step ${\tt FWstep}(x;g_\calP)\mapsto x'$ for $g_\calP\in\mathfrak{F}$ is good then $g_\calP^\downarrow(x')\le \theta(\Theta_{g_\calP}) \cdot g_\calP^\downarrow(x)$;
(ii) when applying ${\tt FWstep}(x;g_\calP)$ iteratively to some initial vector $x_0$ (possibly for different functions $g_\calP\in\mathfrak{F}$),
at any point we have $N_{\tt bad}\le R_0+R_1 N_{\tt good}$ where $N_{\tt good}$ and $N_{\tt bad}$ are respectively the numbers of good and bad steps.

We will consider two classes of functions:
\begin{itemize}
\item $\mathfrak{F}_{\tt strong}=\{g_\calP(x)=g(x)+\delta_\calP(x)\::\:$  $g$ is a strongly convex
differentiable function with a Lipschitz-continuous gradient, with $\Theta_{g_\calP}=(\mu_g,L_g)$  $\}$.
\item $\mathfrak{F}_{\tt weak}=\{g_\calP(x)=g(Ex)+\langle b,x\rangle+\delta_\calP(x)\::\:$ $g$ is a strongly convex
differentiable function with a Lipschitz-continuous gradient, and $E,b$ are matrix and vector of appropriate dimensions,
with $\Theta_{g_\calP}=(\mu_g,L_g,E,b)$ $\}$.
\end{itemize}
Note that class $\mathfrak{F}_{\tt strong}$ is implicitly parameterized by the dimension of vector $x$,
and class $\mathfrak{F}_{\tt weak}$ is implicitly parameterized by the dimensions of vector $x$ and matrix $E$.

%Note that $\mathfrak{F}_{\tt strong}\subseteq \mathfrak{F}_{\tt weak}$.
The AFW method is known to have linear convergence on $\mathfrak{F}_{\tt strong}$~\cite{AFW}
and also on $\mathfrak{F}_{\tt weak}$~\cite{BeckShtern:17,AFW}.
From the result of~\cite{BeckShtern:17,AFW} it is easy to deduce that the DiCG method with away steps
also has linear convergence on $\mathfrak{F}_{\tt weak}$, using~\cite[Property 1]{BashiriZhang:17}.
The BCG method~\cite{BCG} has been shown to have linear convergence on class $\mathfrak{F}_{\tt strong}$.
\begin{remark}
Some of the techniques above maintain some additional information about current iterate $x$.
In particular, AFW and BCG represent $x$ as a convex combination of ``atoms'' (vertices of $\calP$): $x=\sum_i \alpha_i a_i$ where $\alpha_i\ge 0$, $\sum_i \alpha_i=1$ and $a_i$ are atoms.
Coefficients $\alpha_i$ are updated together with $x$. For brevity, we omitted this from the notation. 
\end{remark}
\begin{remark}
The claims about the number of bad steps are proven in~\cite{AFW,GarberMeshi:16,BashiriZhang:17,BCG}
assuming that the function $g_\calP$ is fixed. However, the proofs only use ``structural'' properties of current iterate $x$;
they are easily extended to the case when $g_\calP$ is changing, as long as $\calP$ is fixed.
\end{remark}

\myparagraph{Iterative application of ${\tt FWstep}$} Procedure ${\tt FWstep}$ can be used in
a natural way to solve problems $x\approx_\varepsilon \argmin_x g_\calP(x)$ up to desired accuracy $\varepsilon$.

\begin{algorithm}[H]
  \DontPrintSemicolon
 % select initial vector $x\in\calP$  \\
  \While{\tt true}
      {
        update $x\leftarrow {\tt FWstep}(x;g_\calP)$ \\
        if ${\tt gap}^{\tt FW}(x;g_\calP)\le \varepsilon$ then return $x$
      }
      \caption{Algorithm ${\tt FW}_\varepsilon(x;g_\calP)$.  \\
      {\bf Output:} vector $x'\approx_\varepsilon\argmin_x g_\calP(x)$. }\label{alg:FWeps}
\end{algorithm}

\begin{proposition}\label{prop:FWlogn}
 Suppose that procedure ${\tt FWstep}$ has a linear convergence rate on class $\mathfrak{F}$ that contains $g_\calP$. Then, \\
(a) The number of good steps made during ${\tt FW}_\varepsilon(x_0;g_\calP)$ satisfies
\begin{equation}\label{eq:Ngood}
N_{\tt good}\le   \log_{1/\theta(\Theta_{g_\calP})} \frac{g^\downarrow_\calP(x_0)}{\min\{\frac 12 L D^2,\frac{1}{2L}\left(\frac{\varepsilon}{D}\right)^2\}}
\end{equation}
where $D$ is the diameter of $\calP$ and $L>0$ is any constant satisfying $L\ge L_g$. \\
(b) Suppose that $g_\calP\in\tilde{\mathfrak{F}}\subseteq\mathfrak{F}$
where $\sup_{g_\calP\in \tilde{\mathfrak{F}},x\in\calP}g^\downarrow_\calP(x)<\infty$, 
$\sup_{g_\calP\in \tilde{\mathfrak{F}}}L_g<\infty$,
and  $\{\Theta_{g_\calP}\:|\:g_\calP\in\tilde{\mathfrak{F}}\}$ is a compact subset of $\mathbb R^p$.
Then $N_{\tt good}=O(\log \frac{1}{\varepsilon})$
where the constant in the $O(\cdot)$ notation depends on $\tilde{\mathfrak{F}}$.
\end{proposition}
\begin{proof}
\myparagraph{(a)}
By the definition of linear convergence, after the given number of good steps we obtain vector $x$ satisfying $g^\downarrow_\calP(x)\le \min\{\frac 12 LD^2,\frac{1}{2L}\left(\frac{\varepsilon}{D}
\right)^2\}$.
By Lemma~\ref{lemma:FWgap1}, such $x$ satisfies ${\tt gap}^{\tt FW}(x;g_\calP)\le \varepsilon$,
and therefore the algorithm will immediately terminate.

\myparagraph{(b)} Since set $\{\Theta_{g_\calP}\:|\:g_\calP\in\tilde{\mathfrak{F}}\}\subseteq\mathbb R^p$ is compact and function $\theta:\mathbb R^p\rightarrow(0,1)$ is continuous,
there exists $\theta^\ast\in(0,1)$ such that $\theta(\Theta_{g_\calP})\le\theta^\ast$ for all $g_\calP\in\tilde{\mathfrak{F}}$.
Thus, all quantities present in~\eqref{eq:Ngood} (except for $\varepsilon$) are bounded by constants for all $g_\calP\in\tilde{\mathfrak{F}}$.
The claim follows.
\end{proof}

\section{First approach: dual proximal point algorithm}\label{sec:firstalg}

The first approach that we consider is a proximal point algorithm
(PPA) applied to the dual problem:
\[
\max_{y\in\calY}\left\{ H(y):= \min_{x \in \calX} \calL(x,y) \right\}.
\]
For a point $\bar y \in \calY$ and a smoothing parameter $\gamma > 0$,
we let
\[
\calL_{\gamma, \bar y}(x,y) = \calL(x,y) - \frac{1}{2\gamma}\norm{y-\bar y}^2,
\]
which can be seen as the original saddle-point problem, but with an
additional proximal regularization on the dual variable. In each iteration, the
PPA solves a maximization problem of the following form:
\[ 
\hat y = \argmax_{y\in\calY} \left\{H_{\gamma,\bar y}(y):=
\min_{x\in\calX} \calL_{\gamma, \bar y}(x,y) =
H(y)-\frac{1}{2\gamma}\norm{y-\bar y}^2\right\}.
\]
Based on our structure, it will be beneficial to first solve for $\hat
x$ and then to solve for $\hat y$ via its proximal map, that is
\[
\hat x = \arg\min_{x\in\calX} \left\{ \calF_{\gamma, \bar y}(x) :=
\max_{y\in\calY} \calL_{\gamma, \bar y}(x,y)\right\}, \quad 
\hat y =
\argmax_{y\in\calY}\calL_{\gamma,\bar y}(\hat x,y)=
 \prox_{\gamma h^*}(\bar y + \gamma K\hat x).
\]
Note that also for the smoothed saddle-point problem, strong duality
holds,
\[
\min_{x\in\calX} \calF_{\gamma, \bar y}(x) = \min_{x\in\calX}\max_{y
  \in \calY} \calL_{\gamma, \bar y}(x,y) = \max_{y \in \calY}
H_{\gamma, \bar y}(y),
\]
and hence each step of the PPA can be equivalently  written as
minimizing the primal-dual gap
\begin{equation}\label{eq:pdgap}
(\hat x, \hat y) = \argmin_{(x,y) \in \calX \times \calY}
\calF_{\gamma, \bar y}(x) - H_{\gamma, \bar y}(y)
\end{equation}

It is a well-known fact that the basic proximal point algorithm can be
accelerated to achieve a $O(1/n^2)$ convergence
rate~\cite{Guler-accelerated,Salzo-Villa-PP}, which follows from the
fact that the PPA can be seen as a steepest descent on the Moreau
envelope (see Appendix~\ref{sec:Moreau}), which has a Lipschitz continuous gradient
and hence can be accelerated using the technique of
Nesterov~\cite{Nesterov83}.

However, based on our general assumptions on the
problem~\eqref{eq:saddle}, we will not be able to solve the proximal
subproblems~\eqref{eq:pdgap} exactly but only up to a certain error
$\varepsilon > 0$ that is
\[
(\hat x,\hat y) \approx_\varepsilon \argmin_{(x,y) \in \calX \times \calY}
  \calF_{\gamma, \bar y}(x) - H_{\gamma, \bar y}(y),
\]
which clearly implies that $\hat x \approx_\varepsilon \argmin_{x \in
  \calX} \calF_{\gamma, \bar y}(x)$ as well as $\hat y
\approx_\varepsilon \argmax_{y \in \calY} H_{\gamma, \bar
  y}(y)$. However, we can still apply the recently proposed inexact
accelerated proximal gradient algorithm of Aujol and
Dossal~\cite{AujolDossal:15}, that can handle such approximation while
still achieving an optimal $O(1/n^2)$ convergence rate on the dual
objective. Note that the original method given
in~\cite{AujolDossal:15} only generates the dual sequence $\{y_n\}$
but in Algorithm~\ref{alg:FISTA} below we also keep the primal
sequence $\{x_n\}$ which is needed to obtain a solution of the
original saddle-point problem~\eqref{eq:saddle}. Therefore, the
algorithm below can also be seen as a generalization for solving
saddlepoint problems.

\begin{algorithm}[H]
  \DontPrintSemicolon choose nonnegative sequences $\{t_n\}$, $\{\varepsilon_n\}$ so that
  $t_1=1$ and $\rho_n\eqdef t_{n-1}^2-t_n^2+t_n>0$ for all $n\ge 2$
  \\ choose initial point $y_0\in\calY$, set $\bar y_0=y_0$
  \\ \For{$n=1,2,\ldots$} {
    \begin{eqnarray}
      (x_n,y_n) &\approx_{\varepsilon_n} & \argmin_{(x,y) \in \calX \times \calY}
      \calF_{\gamma, \bar y_{n-1}}(x) - H_{\gamma, \bar y_{n-1}}(y)\\
      \bar y_n &=& y_n + \tfrac{t_n-1}{t_{n+1}} (y_n-y_{n-1})
    \end{eqnarray} 
  }
  \caption{Approximate accelerated proximal gradient
    algorithm.}\label{alg:FISTA}
\end{algorithm}

\addtocounter{equation}{-2}
\refstepcounter{equation}
\label{eq:xy:update}
\refstepcounter{equation}
\label{eq:bary:update}

In order to analyze this algorithm, let us introduce the following quantities:
\begin{align}
u_0&=y_0 , \qquad u_n= y_{n-1}+t_n(y_n-y_{n-1}) \qquad\forall n\ge 1 \label{eq:un:def} \\
A_n&=\sum_{k=1}^n t_k\sqrt{2\gamma \varepsilon_k} \\
B_n&=\sum_{k=1}^n \gamma t_k^2\varepsilon_k \\
W_n&=t_n^2[H(y^\star)-H(y_n)] + \sum_{k=2}^n \rho_k [H(y^\star)-H(y_{k-1})] \\
T_n&= t_n^2 + \sum_{k=2}^n \rho_k = \sum_{k=1}^n t_k 
\end{align}
First, we recall the following result from~\cite{AujolDossal:15}.
\begin{theorem}[\cite{AujolDossal:15}]\label{th:PPA:orig}
For any $n\ge 1$ there holds 
\begin{eqnarray}
W_n + \frac{1}{2\gamma}||u_n-y^\star||^2 
&\le& \frac{C^\ast_n}{2\gamma}
\end{eqnarray}
where
\begin{eqnarray}
C^\ast_n&=&||y_0-y^\star||^2+2A_n\left(||y_0-y^\star||+2A_n+\sqrt{2B_n}\right)+2B_n \\
&\le& \left(||y_0-y^\star||+2A_n+\sqrt{2B_n}\right)^2 
\end{eqnarray}
\end{theorem}
This theorem immediately implies the following results.  (Note, some
of the statements below are slightly modified versions of statements
from~\cite{AujolDossal:15}, but follow exactly the same proofs).
\begin{corollary}[\cite{AujolDossal:15}]\label{cor:PPA:orig}
Suppose that sequences $\{A_n\}$ and $\{B_n\}$ are bounded. 
Then \\
(a) $H(y^\star)-H(y_n)=O(1/t_n^2)$. \\
(b)  $H(y^\star)-H(y^e_n)=O(1/T_n)$ where $y^e_n=(t_n^2y_n+\sum_{k=2}^n\rho_k y_{k-1})/T_n$. \\
(c) If function $-H(y)$ is coercive then sequence $\{y_n\}$ is bounded, and $||y_n-y_{n-1}||=O(1/t_n)$.
\end{corollary}

Note that the rate of convergence in Corollary~\ref{cor:PPA:orig}
depends on the choice of sequence $\{t_n\}$. These are some of the
choices that have appeared in the literature:
\begin{itemize}
\item PPA: $t_n=1$ for all $n\ge 1$. Then
  $T_n=\Theta(n)$.
\item Nesterov~\cite{Nesterov83,FISTA}: $t_{n+1}=(1+\sqrt{1+4t_n^2})/2$ for $n\ge
  1$. Then $t_n=\Theta(n)$ and $T_n=\Theta(n^2)$.
\item Aujol-Dossal~\cite{AujolDossal:15}: $t_n=\left(\frac{n+a-1}{a}\right)^d$ with
  $d\in(0,1]$ and $a>\max\{1,(2d)^{1/d}\}$.  Then $t_n=\Theta(n^d)$
  and $T_n=\Theta(n^{d+1})$.  
%  This is the choice that we will use in our experiments and theorems below.
\end{itemize}
When stating complexities, we will implicitly assume below that either the second case or the third case with $d=1$ is used,
meaning that $t_n=\Theta(n)$ and $T_n=\Theta(n^2)$.

We now generalize Theorem~\ref{th:PPA:orig} to the situation in this
section. This generalization is somewhat analogous to the generalization obtained by Tseng~\cite{Tseng}
(for a different Nesterov-type algorithm and with a different proof).
\begin{theorem}\label{th:PPA}
 Denote $x^e_n=\sum_{k=1}^n t_k x_k / T_n$. For any $y\in\calY$ and any
$n\ge 1$ there holds
\begin{eqnarray}
T_n \left[\calL(x^e_n,y)-H(y^\star)\right]+W_n +
\frac{1}{2\gamma}||u_n-y||^2 &\le& \frac{C_n(y)}{2\gamma}
\end{eqnarray}
where
\begin{eqnarray}\label{eq:Cny}
C_n(y)&=&||y_0-y||^2+2A_n\left(||y-y^\star||+||y_0-y^\star||+2A_n+\sqrt{2B_n}\right)+2B_n
\end{eqnarray}
\end{theorem}
Note that setting $y=y^\star$ in Theorem~\ref{th:PPA} recovers
Theorem~\ref{th:PPA:orig}, since in this case we have
$\calL(x^e_n,y^\star)\ge H(y^\star)$ and $C_n(y^\star)=C^\ast_n$.

\begin{corollary}\label{cor:PPA:new}
Suppose that sequences $\{A_n\},\{B_n\}$ are bounded and $\dom
h^\ast\subseteq\calY$ is a compact set. Then
$\calF(x^e_n)-\calF(x^\star)=O(1/T_n)=O(1/n^2)$.
\end{corollary}
\begin{proof}
By the assumption of the corollary, quantity $C_n(y)$ is bounded for
any $y\in \dom h^\ast$ and $n\ge 1$.  We also have
$\calF(x^e_n)=\max\limits_{y\in\dom h^\ast}\calL(x^e_n,y)$ and
$\calF(x^\star)=H(y^\star)$.  The claim now follows directly from
Theorem~\ref{th:PPA}.
\end{proof}

Next, we analyze the special case of problem~\eqref{eq:saddle} corresponding to constrained optimization problem $\mbox{$\min_{x\in\mathbb R^d}\{f_\calP(x)\:|\:Ax=b\}$}$.
% obtain convergence rates for the following special case of problem~\eqref{eq:saddle}.
%Next, we consider an important special case of~\eqref{eq:saddle} corresponding t
\begin{theorem}\label{th:LinearConstraint}
Suppose we are in the case of the saddle  problem in eq.~\eqref{eq:saddle:Axb}. (a) There holds
%Suppose the Lagrangian in eq.~\eqref{eq:saddle} is given by $\calL(x,y)=f_\calP(x)+\langle y,Ax-b\rangle$,
%i.e. it corresponds to the Lagrangian relaxation
%of problem $\min_{x\in\calX}f_\calP(x)$ subject a linear constraint $Ax=b$.
%Then %part (a) implies that
\begin{eqnarray}
f_\calP(x^e_n)-f_\calP(x^\star)&\le& \frac{C_n(0)}{2\gamma T_n} \label{eq:LinearConstraint:a} \\
||Ax^e_n-b||&\le& \sqrt{\frac{2\max\{f_\calP(x^\star)-f_\calP(x^e_n),0\}}{\gamma T_n}} + \frac{\hat C_n}{\gamma T_n} \label{eq:LinearConstraint:b}
\end{eqnarray}
where
$$
\hat C_n=
||y_0||+A_n
+\sqrt{||y_0||^2+2A_n\left( ||y^\star||+||y_0-y^\star||+2A_n+\sqrt{2B_n}\right)+2B_n}
$$
(b) 
There exists constant $\beta\ge 0$ such that for any $x\in\calP$ 
%there exists $\tilde x\in\calP$ with $||A\tilde x-b||=0$ such that $||x-\tilde x||\le ||Ax-b||$. \\
we have $f_\calP(x^\star)-f_\calP(x)\le \beta ||Ax-b||$. \\ 
(c) 
If sequences $\{A_n\}$ and $\{B_n\}$ are bounded then $f_\calP(x^e_n)-f_\calP(x^\star)=O(1/n^2)$
and $||Ax^e_n-b||=O(1/n^2)$.
\end{theorem}
Note that part (a) is derived directly from Theorem~\ref{th:PPA}. 
In part (b) we crucially exploit the facts that $\calP$ is a polytope, the feasible set $\{x\in\calP:Ax-b=0\}$ is non-empty,
and function $f$ has a bounded gradient on $\calP$. Part (c) is an easy consequence of parts (a) and (b).

\subsection{Overall algorithm}
In this section we fix $n$, and denote $\bar y=\bar y_{n-1}$ and $\varepsilon=\varepsilon_n$.
In order to implement Algorithm~\ref{alg:FISTA} for solving the
saddle-point problem~\eqref{eq:saddle}, we need to specify how to solve
subproblem~\eqref{eq:xy:update} for vector $\bar y$ up to
accuracy $\varepsilon$:
\begin{eqnarray}\label{eq:GNALKSFA}
      (x_n,y_n) &\approx_\varepsilon & \argmin_{(x,y) \in \calX \times \calY}
      \calF_{\gamma, \bar y}(x) - H_{\gamma, \bar y}(y)
\end{eqnarray}
We will first compute $x_n\approx_\varepsilon\argmin_{x\in\calX}\calF_{\gamma,\bar y}(x)$ by invoking Algorithm~\ref{alg:FWeps} for function
$\calF_{\gamma,\bar y}$, and then solve for $y_n$  via its proximal map:
$$
x_n={\tt FW}_\varepsilon(x_{n-1};\calF_{\gamma,\bar y})\qquad\qquad 
y_n=\argmax_{y\in\calY}\calL_{\gamma,\bar y}(x_n,y)=\prox_{\gamma h^*}(\bar y + \gamma K x_n)
$$
 (As we will see later, function $\calF_{\gamma,  \bar y}$ has the form $\calF_{\gamma,  \bar y}(x)=g(x)+\delta_\calP(x)$
for some differentiable convex function $g$ with a $L_g$-Lipschitz continuous gradient, and so Algorithm~\ref{alg:FWeps}
is indeed applicable). 
By construction, vector $x_n$ satisfies ${\tt gap}^{\tt FW}(x_n;\calF_{\gamma,\bar y})\le \varepsilon$.
The following lemma thus implies that the pair $(x_n, y_n)$ indeed solves problem~\eqref{eq:GNALKSFA}.
\begin{lemma}\label{lemma:FWgap2}
Suppose that $\hat x\in\calP$ and $\hat y=\argmax_y \calL_{\gamma,\bar y}(\hat x,y)$.
Then $H_{\gamma,\bar y}(\hat y)\ge \calL_{\gamma,\bar y}(\hat x,\hat y)-\varepsilon=\calF_{\gamma,\bar y}(\hat x)-\varepsilon$ where $\varepsilon={\tt gap}^{\tt FW}(\hat x;\calF_{\gamma,\bar y})$.
\end{lemma}

Next, we derive an explicit expression for function $\calF_{\gamma,\bar y}$ 
(which is needed for implementing the call $x_n={\tt FW}_\varepsilon(x_{n-1};\calF_{\gamma,\bar y})$),
and formulate sufficient conditions on $\calL$ that will guarantee that $\calF_{\gamma,\bar y}\in\mathfrak{F}_{\tt weak}$
(this would yield a good bound on the complexity of Algorithm~\ref{alg:FISTA}).

Recall that the function $\calF_{\gamma,
  \bar y}(x)$ is given by
\[
\calF_{\gamma, \bar y}(x) = \max_{y\in\calY} \calL_{\gamma,\bar
  y}(x,y) = \max_{y\in\calY} \scp{Kx}{y} + f_{\calP}(x)
-h^\ast(y)-\frac{1}{2\gamma}\norm{y-\bar y}^2.
\]
We now show that it can be
written as
\[
\calF_{\gamma, \bar y}(x) = f_{\calP}(x) + h_{\gamma,\bar y}(Kx)
\]
with $h_{\gamma,\bar y}(Kx)$ a differentiable function with Lipschitz
continuous gradient.
\begin{lemma}\label{lemma:h-gamma}
  Let the function $h_{\gamma,\bar y}(Kx)$ be defined as
  \[
  h_{\gamma,\bar y}(Kx) = \max_{y\in\calY} \scp{Kx}{y}
  -h^\ast(y)-\frac{1}{2\gamma}\norm{y-\bar y}^2.
  \]
  We have the following two representations:
  \begin{eqnarray}\label{eq:h-mor}
    h_{\gamma,\bar y}(Kx) &=& \frac{\gamma}{2}\norm{Kx}^2 +
    \scp{Kx}{\bar y} - m_{h^*}^\gamma(\bar y + \gamma Kx),\\
    &=& m_h^{\gamma^{-1}}\left(\gamma^{-1} \bar y + Kx\right) - \frac{1}{2\gamma}\norm{\bar y}^2.
  \end{eqnarray}
  where $m_{h^*}^\gamma$ is the Moreau envelope of $h^*$ with
  smoothing parameter $\gamma$ and $m_h^{\gamma^{-1}}$ is the Moreau
  envelope of $h$ with smoothing parameter $\gamma^{-1}$. \\ Moreover,
  the function $h_{\gamma,\bar y}(Kx)$ is convex, continuously
  differentiable in $x$ with a $\gamma L_K^2$-Lipschitz continuous
  gradient given by
  \begin{eqnarray}\label{eq:h-grad}
    \nabla_x h_{\gamma,\bar y}(Kx) &=& K^*\prox_{\gamma h^*}\left(\bar
    y + \gamma Kx\right),\\ &=& \gamma K^* \left(\gamma^{-1}\bar y +
    Kx - \prox_{\gamma^{-1}h}\left(\gamma^{-1}\bar y + Kx\right)\right).
  \end{eqnarray}
\end{lemma}
In practical applications, we will mostly be interested in the situation
where $h^*$ is a linear constraint.
\begin{lemma}\label{lemma:constraint}
  Let $\calX=\calY=\mathbb R^\ell$, and let $h^*(y) = \delta_{S}(y)$ be a linear constraint of the form $S =
  \{y \in \mathbb R^\ell: Cy=d\}$ for some matrix $C$ with full row rank and
  vector $d$. Then, the function $h_{\gamma,\bar y}(Kx)$ is a
  quadratic function of the form
  \[
  h_{\gamma,\bar y}(Kx) = \frac{\gamma}{2}\norm{Kx}^2 + \scp{Kx}{\bar y} -
  \frac{1}{2\gamma}\norm{C^*(CC^*)^{-1}(d-C(\bar
    y + \gamma Kx))}^2.
  \]
  Moreover, its gradient is a linear map given by
  \[
  \nabla_x h_{\gamma,\bar y}(Kx)  = 
  K^*\left( \bar y + \gamma Kx + C^*(CC^*)^{-1}(d-C(\bar y + \gamma
  Kx))\right).
  \]
\end{lemma}
We therefore obtain the following sufficient condition where the
functions $\calF_{\gamma, \bar y}(x)$ for any $\bar y \in \calY$ and
any $\gamma > 0$ fall into the class $\mathfrak{F}_{\tt weak}$.
\begin{lemma}\label{lemma:F-weak}
  Let $f(x)$ be a quadratic function of the form $f(x) =
  \frac12\scp{Qx}{x} + \scp{q}{x}$, with a symmetric positive
  semidefinite matrix $Q$ and vector $q$. Furthermore, let $h^*$
  satisfy the condition of Lemma~\ref{lemma:constraint}. Then
%  $\calF_{\gamma, \bar y}(x) \in \mathfrak{F}_{\tt weak}$.
  $\{\calF_{\gamma, \bar y}\:|\:\bar y\in\calY\} \subseteq \mathfrak{F}_{\tt weak}$.
\end{lemma}
\begin{proof}
  First note that both $f(x)$ and $ h_{\gamma,\bar y}(Kx)$ are
  quadratic functions. By completing the squares and ignoring constant
  terms, it follows that $\calF_{\gamma, \bar y}(x)$ can be written as
  $\calF_{\gamma, \bar y}(x) = \frac12 \norm{Ex}^2 + \scp{b}{x}
  +\delta_{\calP}(x)$ for some matrix $E$ and vector $b$, where matrix $E$ may depend on $\gamma$ but not on $\bar y$.
\end{proof}

It remains to specify how to set the sequence $\{\varepsilon_n\}$. We
want sequences $\{A_n\}$ and $\{B_n\}$ to be bounded; this can be
achieved by setting $\varepsilon_n=\Theta(n^{-4-\delta})$ for
some $\delta>0$.  With these choices, we obtain the main result of
this section:
\begin{theorem}
Suppose that function $\calL$ satisfies the precondition of
Lemma~\ref{lemma:F-weak}, function $-H(y)$ is coercive, and procedure
${\tt FWstep}$ has a linear convergence rate on $\mathfrak{F}_{\tt weak}$
(e.g.\ it is one step of AFW or DiCG). Then Algorithm~\ref{alg:FISTA}
makes $O(n\log n)$ calls to ${\tt FWstep}$ during the first $n$
iterations, and obtains iterates $x_n^e$ and $y_n^e$ satisfying
%$\calF(x^e_n)-\calF(x^\star)=O(1/n^2)$ (if $\dom h^\ast$ is a compact set)
%and
 $H(y^\star)-H(y^e_n)=O(1/n^2)$.
Furthermore, in the case of the problem in eq.~\eqref{eq:saddle:Axb}
the iterates satisfy $f_\calP(x^e_n)-f_\calP(x^\star)=O(1/n^2)$
and $||Ax_n^e-b||=O(1/n^2)$.
\end{theorem}
\begin{proof}
By Lemma~\ref{lemma:F-weak}, all functions $\calF_{\gamma,\bar
  y_{n-1}}$ encountered during the algorithm belong to
$\mathfrak{F}_{\tt weak}$. Furthermore, vectors $\bar y_{n-1}$ for
$n\ge 1$ belong to a compact set, since the sequence $\{y_n\}$ (and
thus the sequence $\{\bar y_{n}\}$) is bounded by
Corollary~\ref{cor:PPA:orig}(c).  By Proposition~\ref{prop:FWlogn}(b)
the number of good FW steps during $n$-th iteration is $O(\log
\frac{1}{\varepsilon_{n}})=O(\log n)$, and during the first $n$
iterations is $\sum_{k=1}^n O(\log k)=O(n\log n)$.  The number of bad
FW steps is thus also $O(n\log n)$ by the definition of linear
convergence and by the fact that the call $x_n={\tt FW}_{\varepsilon_n}(x_{n-1};\calF_{\gamma,\bar y_{n-1}})$
is initialized with vector $x_{n-1}$.
The remaining claims follow from Corollaries~\ref{cor:PPA:orig} and~\ref{cor:PPA:new} and Theorem~\ref{th:LinearConstraint}.
\end{proof}

\section{Second approach: primal-dual proximal algorithm}\label{sec:secondalg}
In this section we consider solving~\eqref{eq:saddle} without the
restriction that $h^*$ is the indicator function of a linear
constraint. Therefore we make use of proximal primal-dual algorithms
such as~\cite{ChaPock-JMIV} which in each step of the algorithm need
to compute proximal maps with respect to $f_\calP$ and $h^*$. By our
problem assumptions, the proximal map with respect to $h^*$ is
tractable but the proximal map with respect to $f_\calP$ requires to
solve for any $\bar x \in \calX$ and $\tau > 0$ an optimization
problem of the form
\[
\prox_{\tau f_\calP}(\bar x) = \argmin_{x \in \calP} f(x) + \frac1{2\tau}\norm{x-\bar x}^2.
\]
We note the obvious fact that each proximal subproblem is a
$\tau^{-1}$-strongly convex function with $L_f + \tau^{-1}$-Lipschitz
continuous gradient over a convex polytope. Hence, it falls into the
class $\mathfrak{F}_{\tt strong}$, on which Frank-Wolfe algorithms
achieve a linear rate of convergence. Similar to the previous section,
we will not be able to solve the subproblems exactly but up to a
certain accuracy $\varepsilon > 0$. We therefore need to resort to the
recently proposed inexact primal-dual algorithm by Rasch and
Chambolle~\cite[Section 3.1]{RaschChambolle:20} which can handle such
inaccuracy. The algorithm adapted to our situation is given below as
Algorithm~\ref{alg:PD}.

\begin{algorithm}[H]
  \DontPrintSemicolon
  choose $\tau,\sigma>0$ such that $\sigma\tau L_K^2<1$\\
  choose initial points $x_0\in\calX$ and $y_0\in\calY$, set $x_{-1}=x_0$ \\
  \For{$n=0,1,\ldots$}
      {      
        \begin{eqnarray} 
          y_{n+1} & = & \prox_{\sigma h^\ast}(y_n+\sigma K (2x_n-x_{n-1})) \\
          x_{n+1} & \approx_{\varepsilon_{n+1}} & \prox_{\tau f_\calP}(x_n-\tau K^\ast y_{n+1})
        \end{eqnarray}
        
      }
      \caption{Inexact primal-dual algorithm. }\label{alg:PD}
\end{algorithm}

\addtocounter{equation}{-2}
\refstepcounter{equation}
\label{eq:PD:y:update}
\refstepcounter{equation}
\label{eq:PD:x:update}

The following result has been shown
in~\cite{RaschChambolle:20}.~\footnote{Part (a) is not formulated
  explicitly as a theorem in~\cite{RaschChambolle:20}, but can be
  found on page 396 before Theorem 2. Part (b) appears as Theorem 2
  in~\cite{RaschChambolle:20}.}

\begin{theorem}[\cite{RaschChambolle:20}]\label{th:PD}~\\
 (a) Define $x^e_n=\frac{1}{n}\sum_{k=1}^n x_k$ and
  $y^e_n=\frac{1}{n}\sum_{k=1}^n y_k$. Then for any $n\ge 1$ and
  $(x,y)\in\calX\times\calY$
%$$ \calL(x_n^e,y^\star)-\calL(x^\star,y_n^e)\le \frac{1}{2\tau
%n}\left(||x^\star-x_0||+\sqrt{\frac{\tau}{\sigma}}||y^\star-y_0||+2A_n+\sqrt{2B_n}\right)^2 $$
$$ \calL(x_n^e,y)-\calL(x,y_n^e)\le
  \frac{1}{n}\left(\frac{||x-x_0||^2}{2\tau}+\frac{||y-y_0||^2}{2\sigma}+\frac{{\tt
      diam}(\calP)}{\tau}A_n+\frac{1}{\tau}B_n\right)
$$
where
\begin{align}
A_n&=\sum_{k=1}^n \sqrt{2\tau \varepsilon_k} & B_n&=\sum_{k=1}^n
\tau\varepsilon_k
\end{align}
(b) If sequences $\{A_n\}$ and $\{B_n\}$ are bounded then there exists
saddle point $(x^\star,y^\star)$ of problem~\eqref{eq:saddle} such
that $x_n\rightarrow x^\star$ and $y_n\rightarrow y^\star$.
\end{theorem}
To solve the subproblem in eq.~\eqref{eq:PD:x:update}, we call 
Algorithm~\ref{alg:FWeps} via
$$ x_{n+1}\leftarrow {\tt FW}_{\varepsilon_{n+1}}(x_n;g_n),\qquad
g_n(x)=f_\calP(x)+\frac{1}{2\tau}||x-(x_n-\tau K^\ast y_{n+1})||^2
$$
To make sequences $\{A_n\}$
and $\{B_n\}$ bounded, we can set
$\varepsilon_n=\Theta(n^{-2-\delta})$ for some $\delta>0$.  With
these choices, we obtain
\begin{theorem}
Suppose that procedure ${\tt FWstep}$ has a linear convergence
rate on $\mathfrak{F}_{\tt strong}$ (e.g.\ it is one step of AFW, DiCG
or BCG). Then Algorithm~\ref{alg:PD} makes $O(n\log n)$ calls to ${\tt
  FWstep}$ during the first $n$ iterations, and obtains iterates
$x_n^e$ and $y_n^e$ satisfying $F(x^e_n)-F(x^\star)=O(1/n)$ (if $\dom
h^\ast$ is a compact set) and $H(y^\star)-H(y^e_n)=O(1/n)$.
\end{theorem}
\begin{proof}
Clearly, all proximal subproblems $g_n$ encountered during the algorithm
belong to $\mathfrak{F}_{\tt strong}$. Furthermore, vectors $\bar
x_{n+1}\eqdef x_n-\tau K^\ast y_{n+1}$ for $n\ge 0$ belong to a
compact set (by Theorem~\ref{th:PD}(b)).  By
Proposition~\ref{prop:FWlogn}(b) the number of good FW steps during
$n$-th iteration is $O(\log \frac{1}{\varepsilon_{n+1}})=O(\log n)$,
and during the first $n$ iterations is $\sum_{k=1}^n O(\log k)=O(n\log
n)$.  The number of bad FW steps is thus also $O(n\log n)$ by the
definition of linear convergence. 

Now assume that $\dom h^\ast$ is a compact set,
and define $x\in\argmin\limits_{x\in\calP} \calL(x,y^e_n)$,
 $y\in\argmax\limits_{y\in \dom h^\ast} \calL(x^e_n,y)$.
We have
%There exist $x\in\calP$, $y\in\dom h^\ast$ such that
$F(x^e_n)-F(x^\star)=\calL(x^e_n,y)-F(x^\star)\le\calL(x^e_n,y)-H(y^e_n)= \calL(x^e_n,y)-\calL(x,y^e_n)$,
so applying Theorem~\ref{th:PD}(a) for $(x,y)$ gives $F(x^e_n)-F(x^\star)=O(1/n)$.

It remains to prove the last claim.
Since the sets $\{x_n\}$ and $\{y_n\}$ are bounded, we can assume w.l.o.g.\ that $\dom h^\ast$ is a compact set
(we can add indicator function $\delta_S(y)$ to $h^\ast(y)$ for a suitable set $S\subseteq\calY$ so that the algorithm 
and values $H(y^\star)$, $H(y^e_n)$ are not affected). Using the same argument as before, we obtain $H(y^\star)-H(y^e_n)=O(1/n)$.
%$H(y^\star)-H(y^e_n)=H(y^\star)-\calL(x,y^e_n)\le F(x^e_n)-\calL(x,y^e_n)= \calL(x^e_n,y)-\calL(x,y^e_n)$,
\end{proof}

%%%%%%%%%%%%%%%%%%%%%%%%%%%%%%%%%%%%%%%%%%%%%%%%%%%%%%%%%%%%%%%%%%%%%%%%%%%%%%%%%%%%%%%%%%%%%%%%%%%%%%%%%%%%%%%%%%%%%%%
%%%%%%%%%%%%%%%%%%%%%%%%%%%%%%%%%%%%%%%%%%%%%%%%%%%%%%%%%%%%%%%%%%%%%%%%%%%%%%%%%%%%%%%%%%%%%%%%%%%%%%%%%%%%%%%%%%%%%%%
%%%%%%%%%%%%%%%%%%%%%%%%%%%%%%%%%%%%%%%%%%%%%%%%%%%%%%%%%%%%%%%%%%%%%%%%%%%%%%%%%%%%%%%%%%%%%%%%%%%%%%%%%%%%%%%%%%%%%%%
%%%%%%%%%%%%%%%%%%%%%%%%%%%%%%%%%%%%%%%%%%%%%%%%%%%%%%%%%%%%%%%%%%%%%%%%%%%%%%%%%%%%%%%%%%%%%%%%%%%%%%%%%%%%%%%%%%%%%%%
%%%%%%%%%%%%%%%%%%%%%%%%%%%%%%%%%%%%%%%%%%%%%%%%%%%%%%%%%%%%%%%%%%%%%%%%%%%%%%%%%%%%%%%%%%%%%%%%%%%%%%%%%%%%%%%%%%%%%%%

%%%%%%%%%%%%%%%%%%%%%%%%%%%%%%%%%%%%%%%%%%%%%%%%%%%%%%%%%%%%%%%%%%%%%%%%%%%%%%%%%%%%%%%%%%%%%%%%%%%%%%%%%%%%%%%%%%%%%%%
%%%%%%%%%%%%%%%%%%%%%%%%%%%%%%%%%%%%%%%%%%%%%%%%%%%%%%%%%%%%%%%%%%%%%%%%%%%%%%%%%%%%%%%%%%%%%%%%%%%%%%%%%%%%%%%%%%%%%%%
%%%%%%%%%%%%%%%%%%%%%%%%%%%%%%%%%%%%%%%%%%%%%%%%%%%%%%%%%%%%%%%%%%%%%%%%%%%%%%%%%%%%%%%%%%%%%%%%%%%%%%%%%%%%%%%%%%%%%%%
%%%%%%%%%%%%%%%%%%%%%%%%%%%%%%%%%%%%%%%%%%%%%%%%%%%%%%%%%%%%%%%%%%%%%%%%%%%%%%%%%%%%%%%%%%%%%%%%%%%%%%%%%%%%%%%%%%%%%%%
%%%%%%%%%%%%%%%%%%%%%%%%%%%%%%%%%%%%%%%%%%%%%%%%%%%%%%%%%%%%%%%%%%%%%%%%%%%%%%%%%%%%%%%%%%%%%%%%%%%%%%%%%%%%%%%%%%%%%%%
%%%%%%%%%%%%%%%%%%%%%%%%%%%%%%%%%%%%%%%%%%%%%%%%%%%%%%%%%%%%%%%%%%%%%%%%%%%%%%%%%%%%%%%%%%%%%%%%%%%%%%%%%%%%%%%%%%%%%%%

\section{Numerical Results}\label{sec:results}

\begin{figure*}[ht!]
  \setcounter{subfigure}{0}
  \centering
  \subfigure[``Vapnik'' ($200\times 150$, $50$ labels)]{\includegraphics[height=0.205\textwidth]{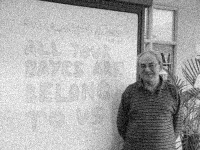}\hfil
    \includegraphics[height=0.205\textwidth]{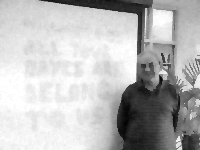}}\hfil
  \subfigure[``Einstein'' ($200 \times 185$, $50$ labels)]{\includegraphics[height=0.205\textwidth]{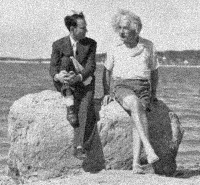}\hfil
    \includegraphics[height=0.205\textwidth]{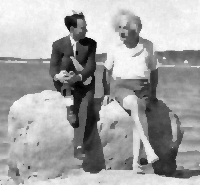}}\\
  \subfigure[``Tsukuba'' ($385\times 288$, $16$ labels)]{\includegraphics[height=0.177\textwidth]{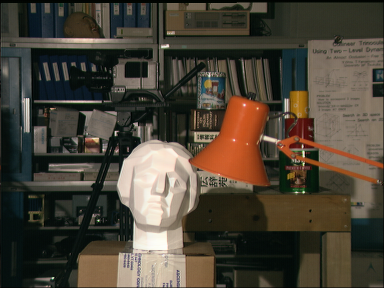}\hfil
    \includegraphics[height=0.177\textwidth]{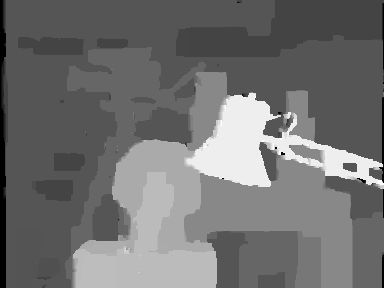}}\hfil
  \subfigure[``Adirondack'' ($718 \times 496$, $64$ labels)]{\includegraphics[height=0.177\textwidth]{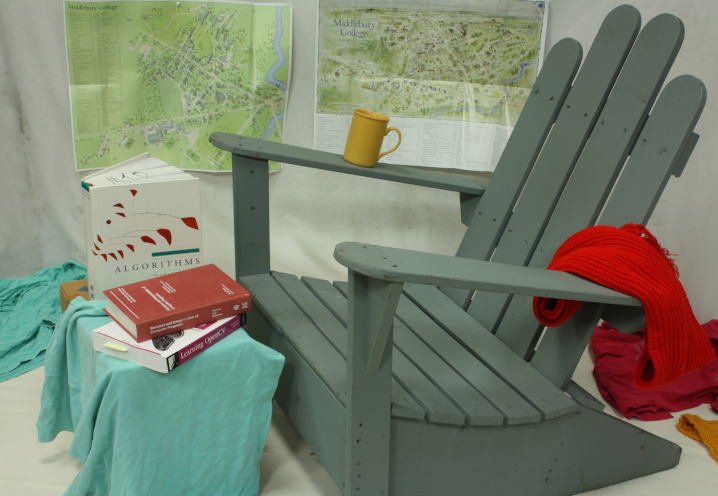}\hfil
  \includegraphics[height=0.177\textwidth]{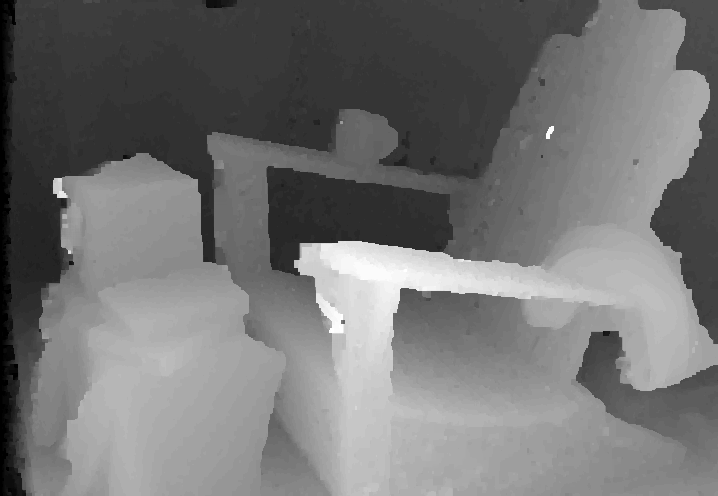}}\\
  \caption{Results of the proposed algorithm for MRF-based image denoising
    (first row) and disparity estimation (second row)}\label{fig:results}
\end{figure*}

\begin{figure*}[ht!]
  \setcounter{subfigure}{0}
  \centering
  \subfigure[]{\includegraphics[width=0.32\textwidth]{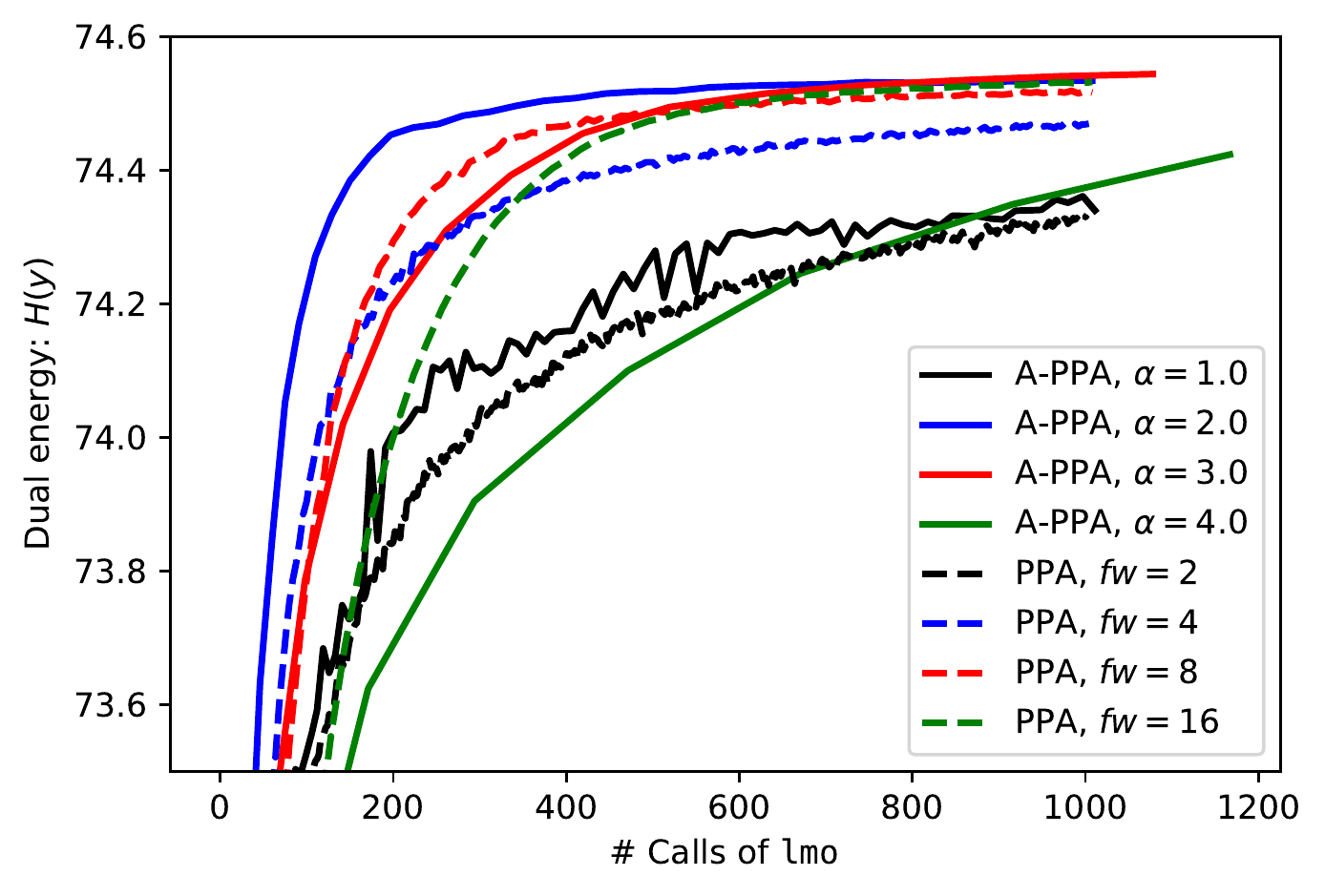}}\hfil
  \subfigure[]{\includegraphics[width=0.32\textwidth]{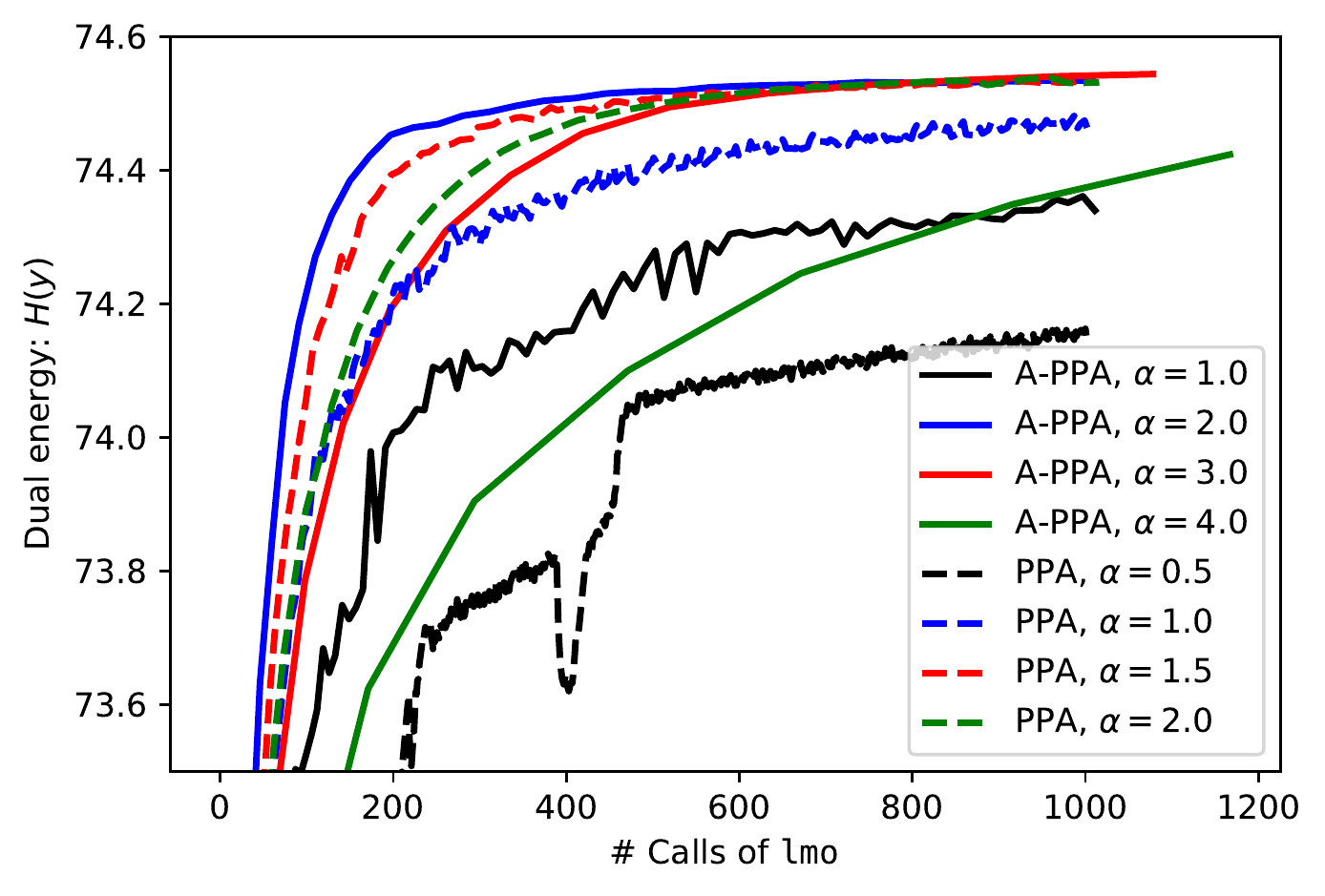}}\hfil
  \subfigure[]{\includegraphics[width=0.32\textwidth]{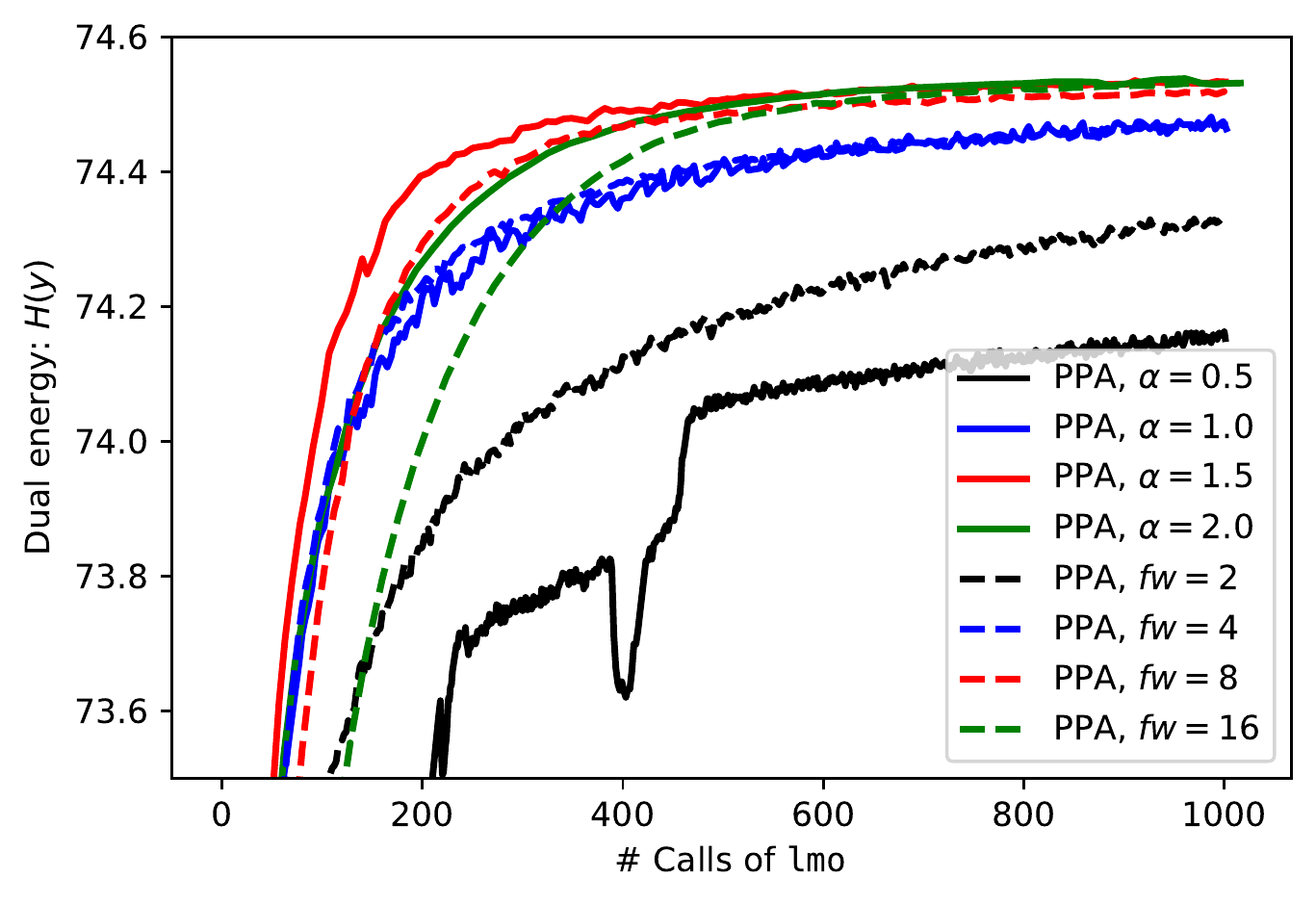}}\\
  \subfigure[]{\includegraphics[width=0.32\textwidth]{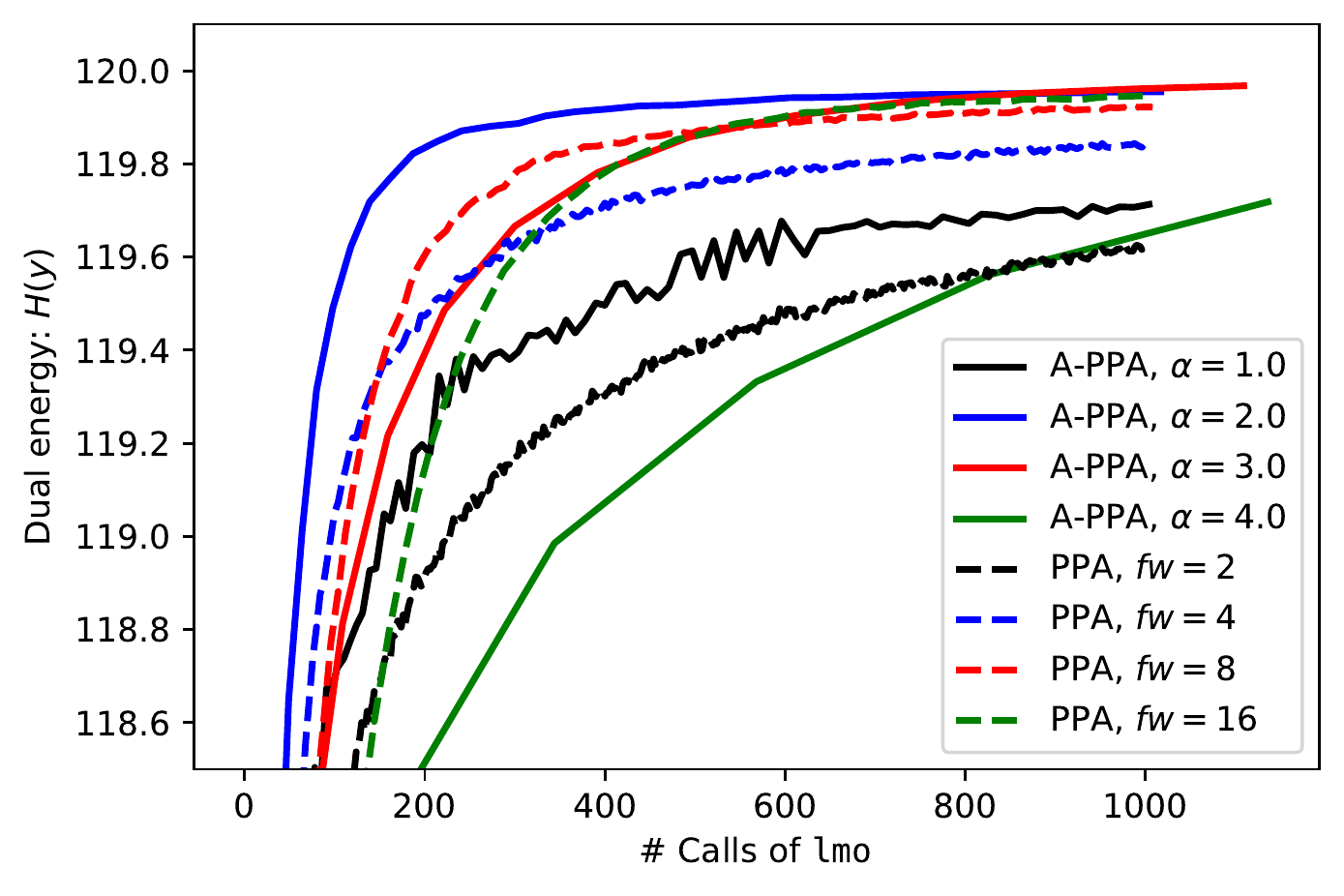}}\hfil
  \subfigure[]{\includegraphics[width=0.32\textwidth]{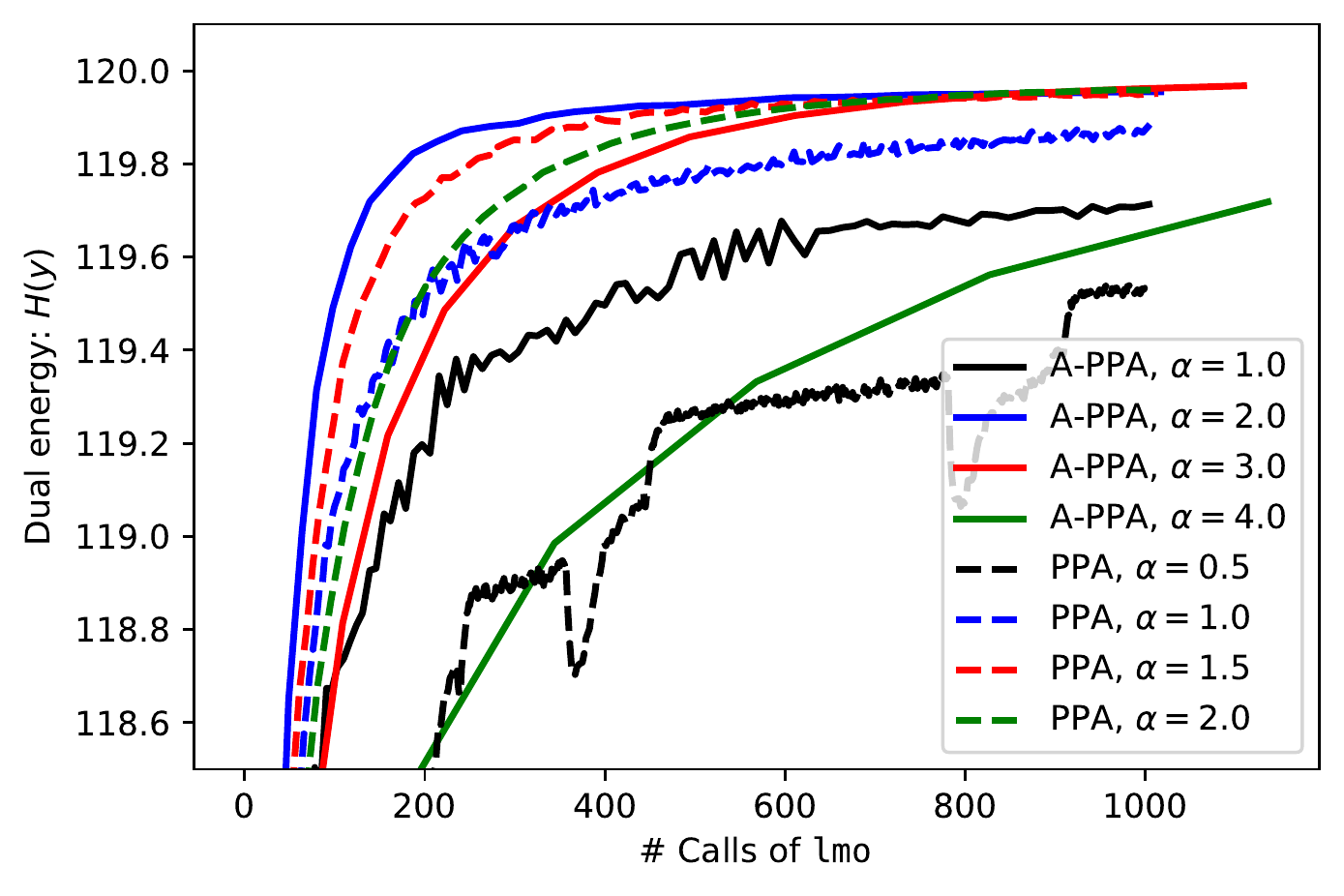}}\hfil
  \subfigure[]{\includegraphics[width=0.32\textwidth]{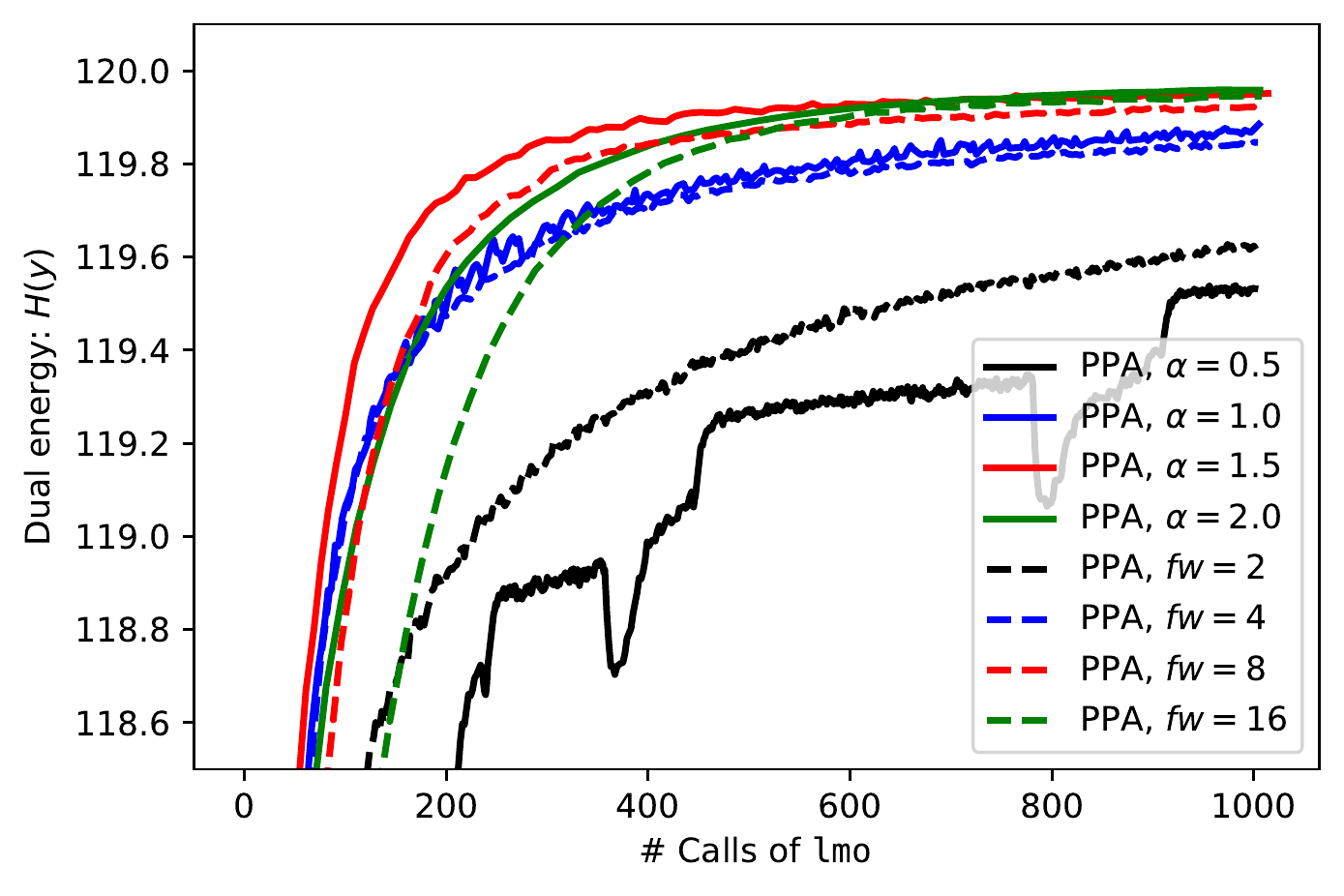}}\\
  \subfigure[]{\includegraphics[width=0.32\textwidth]{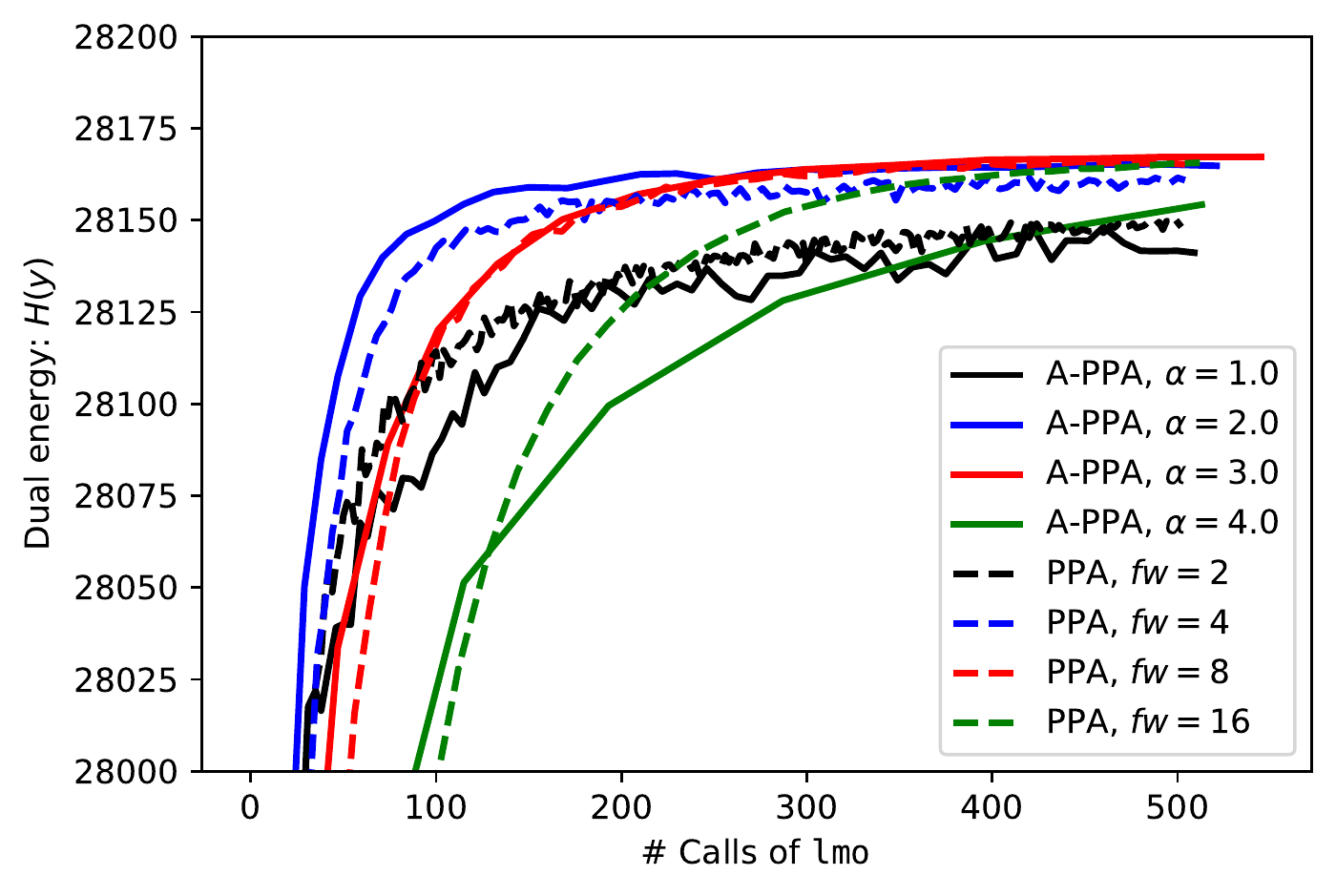}}\hfil
  \subfigure[]{\includegraphics[width=0.32\textwidth]{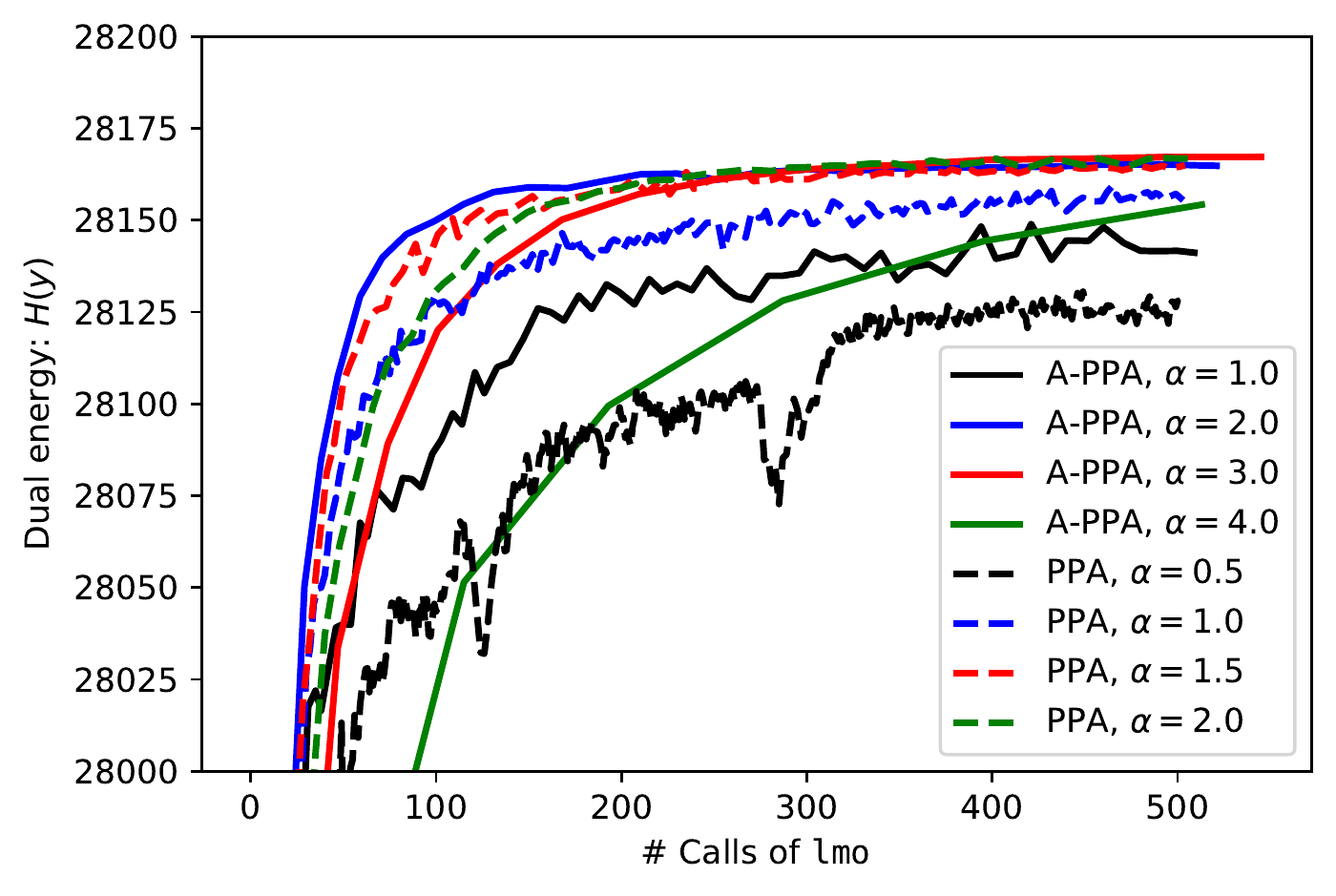}}\hfil
  \subfigure[]{\includegraphics[width=0.32\textwidth]{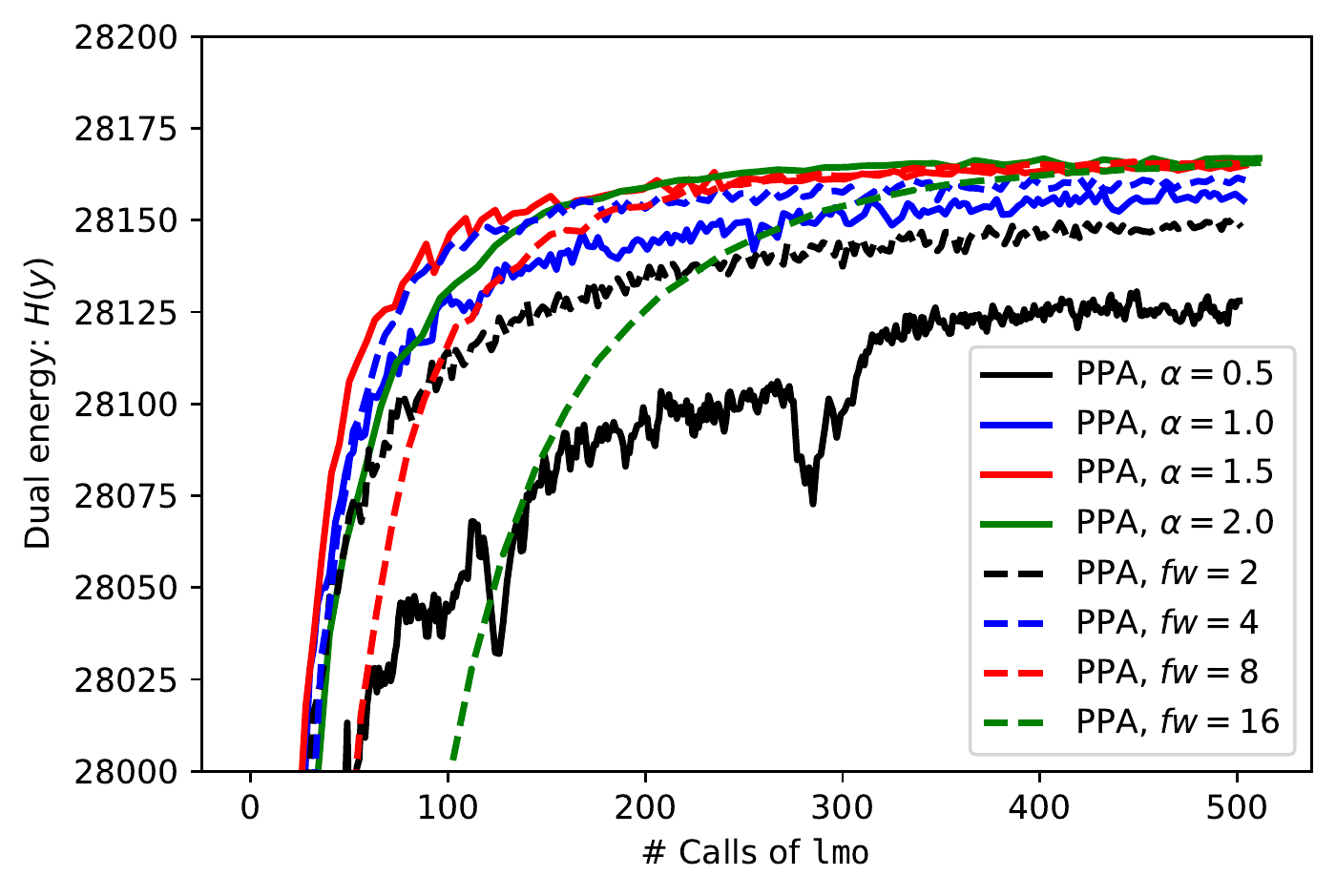}}
  \subfigure[]{\includegraphics[width=0.32\textwidth]{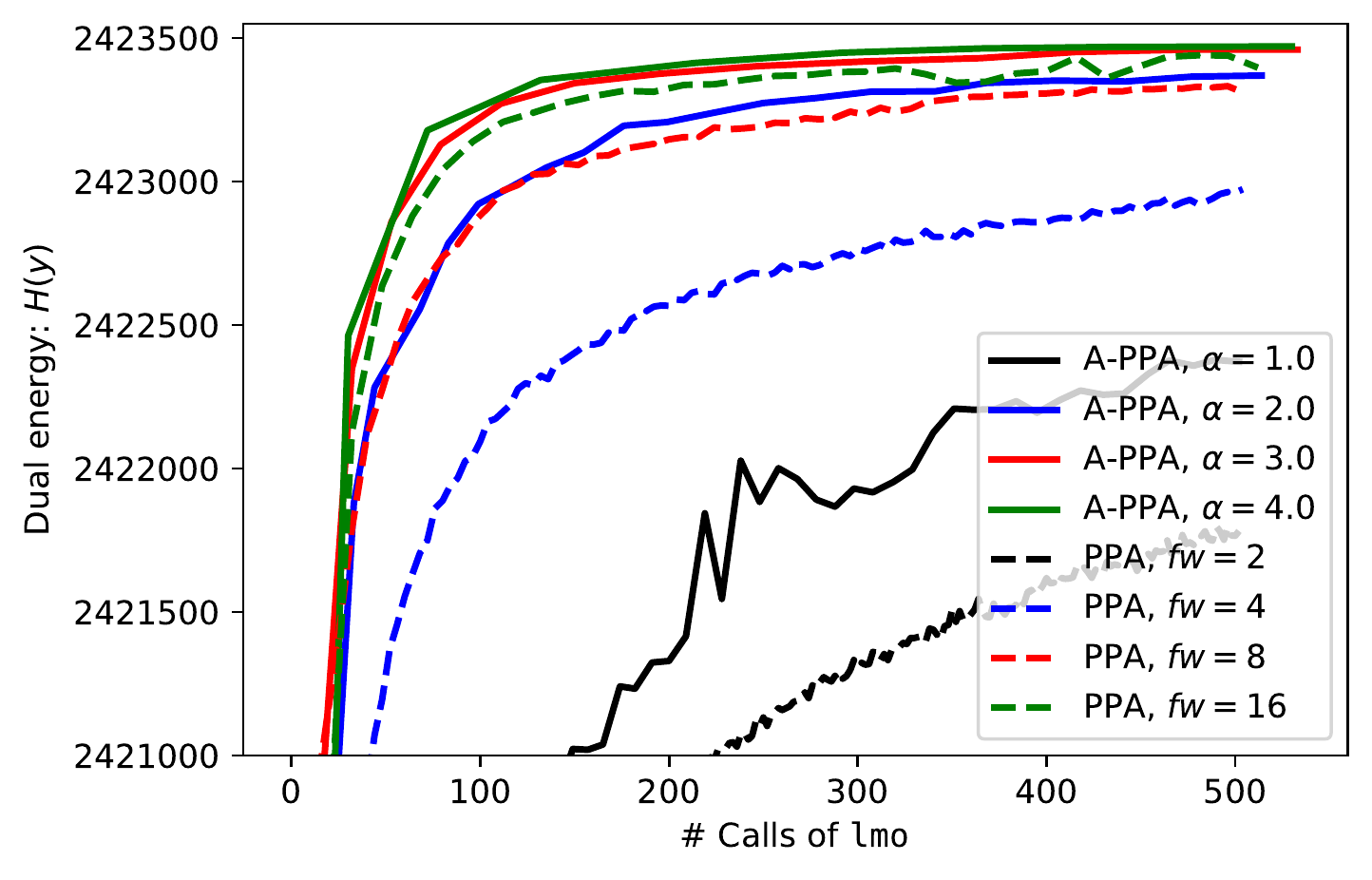}}\hfil
  \subfigure[]{\includegraphics[width=0.32\textwidth]{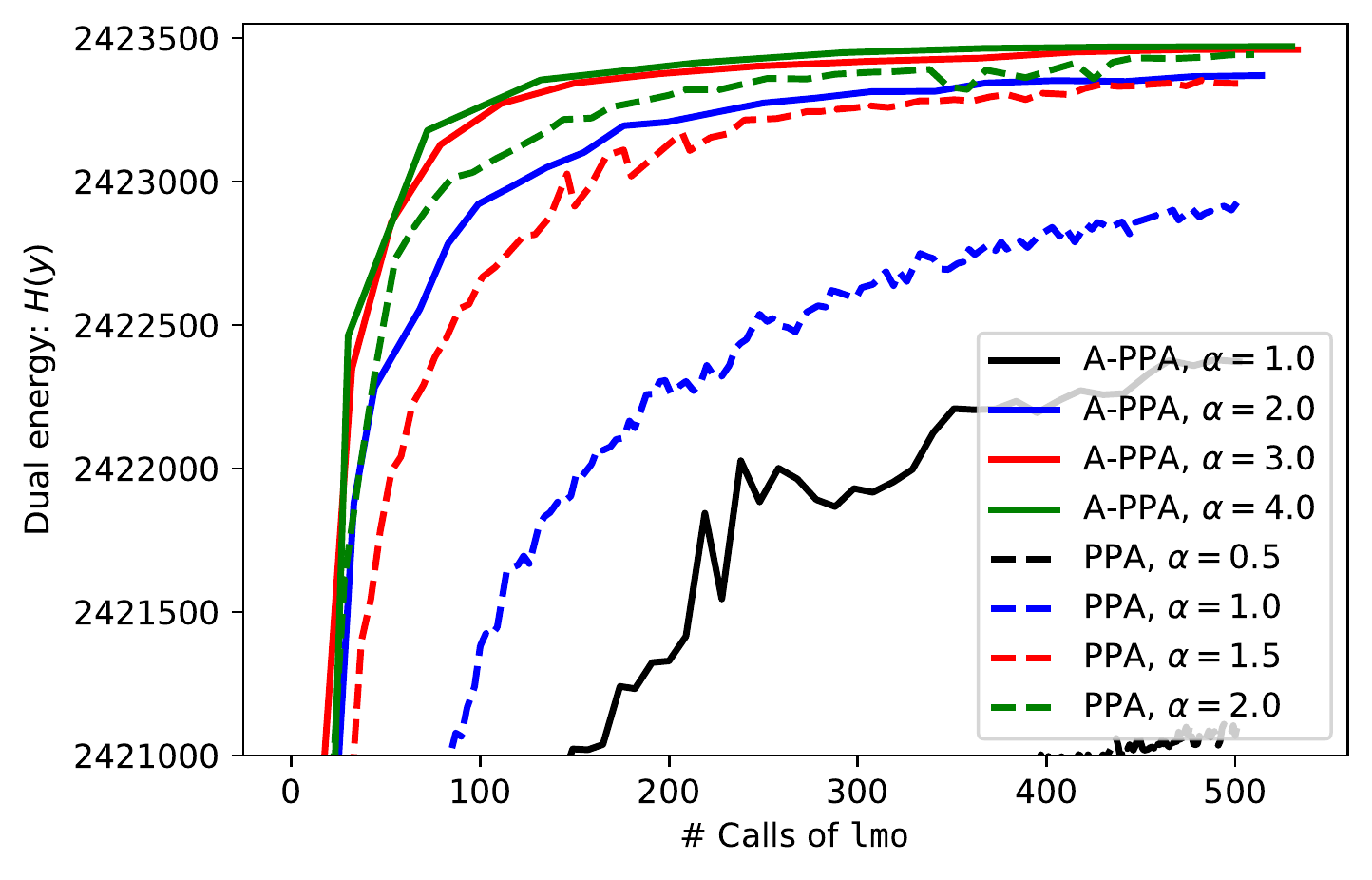}}\hfil
  \subfigure[]{\includegraphics[width=0.32\textwidth]{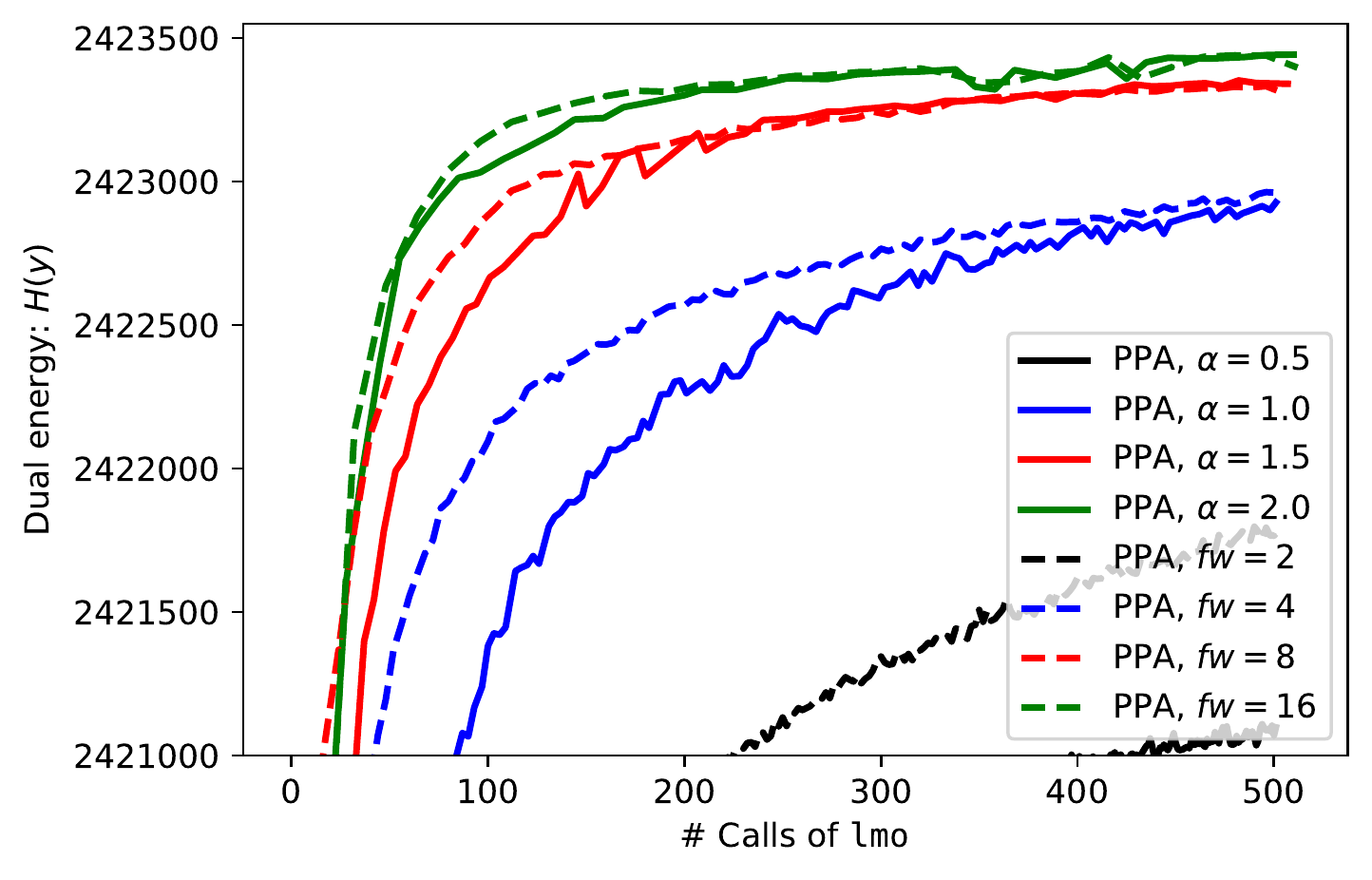}}\\
  \caption{Convergence plots of the dual energy for ``Vapnik'' (first
    row), "Einstein'' (second row), ``Tsukuba'' (third row) and
    ``Adirondack'' (fourth row). See the text for more details.}\label{fig:comparison}
\end{figure*}

In this section, we show preliminary results for solving MRFs arising
from computer vision. The goal is to solve the following discrete
minimization problem:
$$ \min_{\mbox{\footnotesize X}\in D^\calV} \;
E(\mbox{X}):=\sum_{i\in\calV} \theta_i(\mbox{X}_i) + \sum_{ij\in\calE}
\theta_{ij}(\mbox{X}_i,\mbox{X}_j)
$$ where $(\calV,\calE)$ is a 4-connected 2D grid graph, $D$ is a
finite set of labels, and $\theta_i(\cdot),\theta_{ij}(\cdot,\cdot)$
are given unary and pairwise costs, respectively.  We decompose the
problem into horizontal and vertical chains, and convert it to the
saddle point problem~\eqref{eq:saddle} as described in
Section~\ref{sec:example}.

We compare two versions of Algorithm~\ref{alg:FISTA}: accelerated
proximal point algorithm (denoted as A-PPA) with the aggressive choice
%$d=1,a=2$ (so that $t_n = (n+1)/2$), 
$t_n = (n+1)/2$ (which corresponds to the Aujol-Dossal scheme with $d=1,a=2$),
and the standard proximal point
algorithm (denoted as PPA) which is realized by setting $t_n=1$.
Their convergence rates after $n$ iterations are $O(1/n^2)$ and
$O(1/n)$ respectively, assuming that sequences $\{A_n\}$ and $\{B_n\}$
are bounded.  We invoke Algorithm~\ref{alg:FWeps} to minimize the
functions $F_{\gamma, \bar y}$ up to accuracy
\begin{equation}\label{eq:epsn}
  \varepsilon_n = {\tt gap}_0 \cdot n^{-\alpha},
\end{equation}
for a constant $\alpha>0$, where ${\tt gap}_0$ denotes the initial gap
of the function $F_{\gamma,\bar y}$. Note that for $O(1/n^2)$
convergence rate for A-PPA we would have to set $\alpha>4$, but in
practice smaller settings of $\alpha$ give larger speedups.

The smoothing parameter $\gamma$ is chosen by hand, but as a general
rule of thumb, it should be adapted to the scale of the primal
respectively dual cost functions. Note that the selection of an
optimal value of $\gamma$ is beneficial mainly from a practical point
of view as it does not influence the asymptotic convergence rate.  We
leave the selection of an optimal value of $\gamma$ for future work.

In the case of PPA we tested two different versions. The first version
is the method proposed in~\cite{Swoboda:CVPR19}. It uses a fixed
number of Frank-Wolfe steps and is hence denoted as PPA-fw. The second
version is motivated from our analysis for the case $t_n=1$. It
uses~\eqref{eq:epsn} for determining the accuracy of the subproblems
and hence we denote it by PPA-$\alpha$. Note that in order to
guarantee an $O(1/n)$ convergence of PPA, one needs $\alpha > 2$ but
as before, smaller values yield larger speedups.

Procedure ${\tt FWstep}$ was implemented as follows. We explicitly
maintain the current iterate $x$ as a convex combination of atoms. In
the beginning of ${\tt FWstep}$ we first run a standard Frank-Wolfe
step, and then re-optimize the objective over the current set of
atoms. For that one needs to solve a low-dimensional strongly convex,
quadratic subproblem over the unit simplex; we used the linearly
converging accelerated proximal gradient method described
in~\cite{Nesterov04}. The resulting method can be viewed as a version
of the BCG method~\cite{BCG}, which is known to be linearly convergent
on $\mathfrak{F}_{\tt strong}$.

\begin{remark}
Note that there is an extensive literature on FW variants.  Potential
alternatives to the method above include BCFW~\cite{BCFW} and its
variants~\cite{MP-BCFW,Osokin:ICML16},
DiCG~\cite{GarberMeshi:16,BashiriZhang:17}, and Frank-Wolfe with
in-face directions~\cite{FW:in-face}.
% or its recursive
%version~\cite{FW:recursive} (which has been shown
%in~\cite{FW:recursive} to achieve a significant speed-up on some
%shortest path problems that are related to MAP-MRF optimization on
%chain graphs).  
However, testing different Frank-Wolfe methods is
outside the scope of this paper; instead, we study what is the best
way to use a given FW method.
\end{remark}

Next, we describe the two vision applications that we used.

\myparagraph{Image denoising} The first application is a simple
Gaussian image denoising problem. The unary terms are given by a
quadratic potential function and the pairwise terms are given by a
truncated quadratic potential function. Hence, this model resembles a
discrete version of the celebrated Mumford-Shah
functional. Figure~\ref{fig:results} (a) and (b) show the noisy input
images together with their denoised versions. The value of $\gamma$
was set to $\gamma=1e-03$ in both cases.

\myparagraph{Stereo} The second application is a classical disparity
estimation problem from rectified stereo image
pairs~\cite{Scharstein2014}. Figure~\ref{fig:results} (c) and (d) show
the left input images together with the estimated disparity
images. For the ``Tsukuba'' image, the unaries are computed based on
the absolute color differences between the left and right input image.
The unaries of the ``Adirondack'' image are computed using a CNN-based
correlation network~\cite{knobelreiter2020belief}. In bother cases the
pairwise costs are given by a truncated absolute potential
function. The value of $\gamma$ was set to $\gamma=1e-01$ in the case
of ``Tsukuba'', and $\gamma=1.0$ in the case of ``Adirondack''.

\myparagraph{Results} In Figure~\ref{fig:comparison} we provide
convergence plots of the three different algorithms: A-PPA, PPA-fw and
PPA-$\alpha$ on all 4 examples. We plot the convergence of the
dual energy $H(y)$ over the total number of calls of $\lmo$.

The plots show that A-PPA systematically outperforms both the baseline
method PPA-fw (first column) as well as its variant PPA-$\alpha$
(second column). The results also indicate that the choice of $\alpha$
significantly influences the global convergence behavior. For example
in Figure~\ref{fig:comparison} (a) $\alpha=2$ yields the largest
speedup in the first iterations but is catched up by $\alpha=3$ in
later iterations of the algorithm. This suggest, that a more flexible
sequence $\varepsilon_n$ has the potential for an additional speedup
but such a development is left for future research.

From the plots in the third column one can also observe that the new
variant PPA-$\alpha$ has a small speedup over the existing baseline
method PPA-fw. This can be seen as a byproduct of our novel analysis
in the case $t_n=1$. Here one can also observe that values of $\alpha
< 2$ yield the fastest convergence in the first iterations of the
algorithm.

From a practical point of view, the plots show that $\alpha$ should be
chosen according to the desired accuracy of the solution. This is of
particular interest if one is only interested in a fast approximate
solution to the problem, for example if the MRF is used as the last
inference layer in a CNN.
\section{Conclusion}

In this work, we have proposed new primal-dual algorithms based on a
mixture of proximal and Frank-Wolfe algorithms to solve a class of
convex-concave saddle point problems arising in Lagrangian relaxations
of discrete optimization problems. As our main result, we have shown
after $O(n \log n)$ calls to $\lmo$ a $O(1/n)$ convergence rate in the
most general case (Alg.~\ref{alg:PD}) and a $O(1/n^2)$ convergence rate with certain
regularity assumptions on the dual objective (Alg.~\ref{alg:FISTA}). To the best of our
knowledge, this improves on the known rates from the literature. Our
preliminary numerical results also show an improved practical
performance of Alg.~\ref{alg:FISTA}
on MAP inference problems in computer vision.
Note, we have not implemented yet Alg.~\ref{alg:PD} since its rate is worse on
the application that we consider; at the moment the primary purpose of Alg.~\ref{alg:PD}
is to show which rates are achievable for different classes of saddle point problems.

%%%%%%%%%%%%%%%%%%%%%%%%%%%%%%%%%%%%%%%%%%%%%%%%%%%%%%%%%%%%%%%%%%%%%%%%
%%%%%%%%%%%%%%%%%%%%%%%%%%%%%%%%%%%%%%%%%%%%%%%%%%%%%%%%%%%%%%%%%%%%%%%%
%%%%%%%%%%%%%%%%%%%%%%%%%%%%%%%%%%%%%%%%%%%%%%%%%%%%%%%%%%%%%%%%%%%%%%%%
%%%%%%%%%%%%%%%%%%%%%%%%%%%%%%%%%%%%%%%%%%%%%%%%%%%%%%%%%%%%%%%%%%%%%%%%
%%%%%%%%%%%%%%%%%%%%%%%%%%%%%%%%%%%%%%%%%%%%%%%%%%%%%%%%%%%%%%%%%%%%%%%%
%%%%%%%%%%%%%%%%%%%%%%%%%%%%%%%%%%%%%%%%%%%%%%%%%%%%%%%%%%%%%%%%%%%%%%%%
%%%%%%%%%%%%%%%%%%%%%%%%%%%%%%%%%%%%%%%%%%%%%%%%%%%%%%%%%%%%%%%%%%%%%%%%
%%%%%%%%%%%%%%%%%%%%%%%%%%%%%%%%%%%%%%%%%%%%%%%%%%%%%%%%%%%%%%%%%%%%%%%%
%%%%%%%%%%%%%%%%%%%%%%%%%%%%%%%%%%%%%%%%%%%%%%%%%%%%%%%%%%%%%%%%%%%%%%%%
%%%%%%%%%%%%%%%%%%%%%%%%%%%%%%%%%%%%%%%%%%%%%%%%%%%%%%%%%%%%%%%%%%%%%%%%
%%%%%%%%%%%%%%%%%%%%%%%%%%%%%%%%%%%%%%%%%%%%%%%%%%%%%%%%%%%%%%%%%%%%%%%%
%%%%%%%%%%%%%%%%%%%%%%%%%%%%%%%%%%%%%%%%%%%%%%%%%%%%%%%%%%%%%%%%%%%%%%%%

\appendix

\section{Some results from convex optimization}\label{sec:Moreau}
In order to make the paper more self-contained, we recall here some
well-known definitions and results from convex optimization, which
will be useful later. The results with their proofs can for example be
found in the recent monograph by Amir Beck~\cite{beck2017first}.

The convex conjugate $f^*$ of an extended real valued function $f$ is
defined as
\begin{equation}\label{eq:conv-conj}
f^*(y) = \sup_x \scp{x}{y} - f(x),
\end{equation}
which, since being a pointwise maximum over linear affine functions,
is a convex lower semicontinuous function.

The infimal convolution between two proper functions $f$ and $g$
(written as $f \square g$) is defined by
\begin{equation}\label{eq:inf-conv}
(f\square g)(x)  = \inf_u f(u) + g(x-u) = \inf_{x=u+v} f(u) + g(v).
\end{equation}
which itself is a convex function, if $f$ is proper and convex, and
$g$ is convex and real valued.

The infimal convolution and the convex conjugate for two proper
functions $f,g$ are related via
\begin{equation}\label{eq:conj-of-inv-conv}
(f \square g)^* = f^* + g^*.
\end{equation}
Moreover, assuming that $f$ is proper and convex and $g$ is a
real-valued convex function, one also has that
\begin{equation}\label{eq:conj-of-sum}
(f + g)^* = f^* \square g^*.
\end{equation}

The Moreau envelope of a convex function $f$ with smoothness parameter
$\mu > 0$ is the infimal convolution of $f$ with $\mu^{-1}$ times a
quadratic function.
\begin{equation}\label{eq:mor-env}
m^\mu_f (x) = \inf_{x=u+v} f(u) + \frac{1}{2\mu}\norm{v}^2 = \left(f
\square \frac{1}{2\mu} \norm{\cdot}^2\right)(x).
\end{equation}
Its unique minimizing argument is precisely given by the proximal map
\begin{equation}\label{eq:prox}
\argmin_u f(u) + \frac{1}{2\mu}\norm{x-u}^2 = \prox_{\mu f}(x),
\end{equation}
so that the Moreau envelope can explicitly be written as
\begin{equation}\label{eq:mor-env-prox}
m^\mu_f (x) =  f(\prox_{\mu f}(x)) + \frac{1}{2\mu}\norm{x-\prox_{\mu f}(x)}^2.
\end{equation}
It is a well-known fact (and it easily follows from the monotonicity
of the subdifferential), that the proximal map is a firmly
nonexpansive operator, that is for any proper, convex lower
semicontinuous function $f$ it holds that
\begin{equation}\label{eq:firm-non-exp}
  \norm{\prox_{f}(x) - \prox_{f}(y)}^2 \leq \scp{\prox_{f}(x) -
    \prox_{f}(y)}{x-y}, \quad \forall x,y.
\end{equation}
Furthermore, from the Cauchy-Schwartz inequality it also follows that
the proximal map is a nonexpansive operator
\begin{equation}\label{eq:non-exp}
\norm{\prox_{f}(x) - \prox_{f}(y)} \leq \norm{x-y}, \quad \forall x,y.
\end{equation}
The Moreau envelop is continuously differentiable with an explicit
representation of the gradient in terms of its proximal map
\begin{align}\label{eq:mor-grad}
\nabla m^\mu_f (x) = \frac1\mu (x -\prox_{\mu f}(x)).
\end{align}
Furthermore, the gradient of the Moreau envelope is Lipschitz
continuous with parameter $\mu^{-1}$.

The next computation (which makes use of some results above) shows the
relation between the Moreau envelopes of a convex function $f$ and its
convex conjugate $f^*$.
\begin{align*}
  m^\mu_f (x) =&
  \inf_u f(u) + \frac{1}{2\mu}\norm{x}^2 - \frac1\mu\scp{x}{u} + \frac{1}{2\mu}\norm{u}^2\\
= &\frac{1}{2\mu}\norm{x}^2 - \frac1\mu \sup_u \scp{x}{u} - \mu f(u) - \frac{1}{2}\norm{u}^2\\
= & \frac{1}{2\mu}\norm{x}^2 - \frac1\mu \left(\mu f(u) + \frac{1}{2}\norm{\cdot}^2 \right)^*(x)\\
= & \frac{1}{2\mu}\norm{x}^2 - \left( f^* \square \frac{\mu}{2}\norm{\cdot}^2 \right)\left(\frac{x}{\mu}\right),
\end{align*}
from which it follows the beautiful relation
\begin{equation}\label{eq:mor-conj}
  \frac{1}{2\mu}\norm{x}^2 = m^\mu_f (x) + m^{\mu^{-1}}_{f^*} \left(\frac{x}{\mu}\right).
\end{equation}
Now, taking the gradient with respect to $x$ on both sides and
multiplying by $\mu$, recovers the celebrated Moreau identity
\begin{equation}\label{eq:mor-id}
x = \prox_{\mu f}(x) + \mu\prox_{\frac1\mu f^*}\left(\frac{x}{\mu}\right).
\end{equation}

\section{Analysis of Algorithm~\ref{alg:FISTA}: proof of Theorem~\ref{th:PPA}}\label{sec:proof:first}

\subsection{Proof preliminaries}

We first state a few technical results the will be used in the proof of Theorem~\ref{th:PPA}.

\begin{lemma}\label{lemma:ALSKFJASF}
Suppose that $(\hat x,\hat y)\approx_\varepsilon\argmin_{x,y}[F_{\gamma,\bar y}(x)-H_{\gamma,\bar y}(y)]$. Then
$\hat y\approx_{\varepsilon-\delta} \argmax_{y} \calL_{\gamma,\bar y}(\hat x,y)$
where $\delta=\calL(\hat x,\hat y)-H(\hat y)\ge 0$.
\end{lemma}
\begin{proof}
The claim follows from the following inequalities:
$$
\max_{y}\calL_{\gamma,\bar y}(\hat x,y)-\calL_{\gamma,\bar y}(\hat x,\hat y)
=F_{\gamma,\bar y}(\hat x)-\calL_{\gamma,\bar y}(\hat x,\hat y)
=F_{\gamma,\bar y}(\hat x)-H_{\gamma,\bar y}(\hat y)-\delta
\le \varepsilon-\delta
$$
\end{proof}

\begin{lemma}[{\cite[Lemma 2]{Schmidt:11},\cite[Lemma 2.5]{AujolDossal:15}}]\label{lemma:GJHAS}
Suppose that $\hat y\approx_{\varepsilon} \prox_{\gamma\phi}(\bar y)$ for a proper convex l.s.c.\ function $\phi$.
Then there exists $r$ with $||r||\le \sqrt{2\gamma\varepsilon}$ such that
$$
\phi(z)\ge \phi(\hat y)+\left\langle \frac{\bar y-\hat y-r}{\gamma},z-\hat y\right\rangle-\varepsilon \qquad\qquad\forall z\in\calY
$$
\end{lemma}

\begin{lemma}\label{lemma:NFLAKSN}
Suppose that $(\hat x,\hat y)\approx_\varepsilon\argmin_{x,y}[F_{\gamma,\bar y}(x)-H_{\gamma,\bar y}(y)]$. Then
$$
H(\hat y)\ge \calL(\hat x,z)+\frac{||\hat y-z||^2-||z-\bar y||^2+||\hat y-\bar y||^2}{2\gamma}-\varepsilon -  \sqrt{\frac{2\varepsilon}{\gamma}} \cdot||\hat y-z||
\qquad\quad\forall z\in\calY
$$
\end{lemma}
\begin{proof}
Denote 
 $\phi(z)=-\calL(\hat x,z)$. By Lemma~\ref{lemma:ALSKFJASF}, we have $\hat y\approx_{\varepsilon-\delta} \prox_{\gamma\phi}(\bar y)$
where $\delta=\calL(\hat x,\hat y)-H(\hat y)\ge 0$.
By Lemma~\ref{lemma:GJHAS}, there exists $r$ with $||r||\le \sqrt{2\gamma(\varepsilon-\delta)}\le\sqrt{2\gamma\varepsilon}$ such that
\begin{eqnarray*}
-H(\hat y)
\;\;=\;\;\delta+\phi(\hat y)
&\le&\delta+\phi(z)-\left\langle \frac{\bar y-\hat y-r}{\gamma},z-\hat y\right\rangle+\varepsilon-\delta \\
&=&-\calL(\hat x,z)-\tfrac{1}{\gamma}\left\langle \bar y-\hat y,z-\hat y\right\rangle  + \tfrac{1}{\gamma}\left\langle r,z-\hat y\right\rangle + \varepsilon \\
&=&-\calL(\hat x,z)+\tfrac{1}{2\gamma}\left(   ||\bar y-z||^2-||\hat y-z||^2 - ||\hat y-\bar y||^2    \right)  + \tfrac{1}{\gamma}\left\langle r,z-\hat y\right\rangle + \varepsilon
\end{eqnarray*}
It remains to observe that $\left\langle r,z-\hat y\right\rangle\le  ||r||\cdot ||\hat y-z||\le \sqrt{2\gamma\varepsilon}\cdot ||\hat y-z||$.
\end{proof}
The following identity will also be useful (it follows from eq.~\eqref{eq:bary:update} and~\eqref{eq:un:def}):
\begin{equation}\label{eq:LAHFASG}
\bar y_n=\left(1-\frac{1}{t_{n+1}}\right)y_n + \frac{1}{t_{n+1}} u_n
\end{equation}

\subsection{Proof of Theorem~\ref{th:PPA}}

\begin{proof}
Let us fix $y\in\calY$ (to be determined later). For brevity, denote $a_k=a_k(y)=||u_k-y||$ for $k\ge 0$.
Now fix $k\ge 1$, and  apply Lemma~\ref{lemma:NFLAKSN} to $\bar y=\bar y_{k-1}$, $\hat y=y_k$ and $z=(1-\frac{1}{t_k})y_{k-1} + \frac{1}{t_k}y$. We get
\begin{equation}\label{eq:POASJGASF}
H(y_k)\ge \calL(x_k,z)+\frac{||y_k-z||^2-||z-\bar y_{k-1}||^2}{2\gamma} -\varepsilon_k - 
\sqrt{\frac{2\varepsilon_k}{\gamma}} \cdot||y-y_k||
\end{equation}
Using~\eqref{eq:LAHFASG}, it can be checked that $y_k-z=\frac{1}{t_k}(u_k-y)$ and $z-\bar y_{k-1}=\frac{1}{t_k}(y-u_{k-1})$.
Furthermore, by concavity of function $\calL(x_k,\cdot)$ we get
\begin{eqnarray*}
\calL(x_k,y)
&\ge& \left(1-\frac{1}{t_k}\right)\calL(x_k,y_{k-1}) + \frac{1}{t_k}\calL(x_k,y) \\
&\ge& \left(1-\frac{1}{t_k}\right)H(y_{k-1}) + \frac{1}{t_k}\calL(x_k,y)
\end{eqnarray*}
Plugging these equations into~\eqref{eq:POASJGASF} and multiplying by $t_k^2$ gives
$$
t_k^2 H(y_k)\ge (t_k^2-t_k)H(y_{k-1}) + t_k\calL(x_k,y) +\frac{a_k^2-a_{k-1}^2}{2\gamma }   -t_k^2\varepsilon_{k} - 
t_k\sqrt{\frac{2\varepsilon_k}{\gamma}}\cdot a_k 
$$
Let us sum this inequality over $k=1,\ldots,n$ and move terms with $H(\cdot)$ to the LHS. We obtain
$$
t_n^2 H(y_n)+\sum_{k=1}^{n-1} \rho_{k+1} H(y_k) \ge \sum_{k=1}^n t_k \calL(x_k,y) + \frac{a_n^2-a_0^2}{2\gamma} -\sum_{k=1}^n t_k^2\varepsilon_k - \sum_{k=1}^n t_k\sqrt{\frac{2\varepsilon_k}{\gamma}}\cdot a_k
$$
Note that the LHS equals $T_nH(y^\star)-W_n$. By convexity of function $\calL(\cdot,y)$, we have $\sum_{k=1}^n t_k\calL(x_k,y)\ge T_n\calL(x^e_n,y)$.
Therefore,
\begin{eqnarray}\label{eq:OAISGA}
T_nH(y^\star)-W_n&\ge& T_n \calL(x^e_n,y) + \frac{||u_n-y||^2}{2\gamma} - \frac{\tilde C_n(y)}{2\gamma} 
\end{eqnarray}
where we denoted
\begin{eqnarray*}
\tilde C_n(y)&=& ||y_0-y||^2 + 2\sum_{k=1}^n \gamma t_k^2\varepsilon_k + 2\sum_{k=1}^n t_k\sqrt{2\gamma\varepsilon_k}\cdot ||u_k-y|| 
\end{eqnarray*}
Equivalently,
\begin{eqnarray}\label{eq:JHBASJBFAKSFAS}
T_n \left[\calL(x^e_n,y)-H(y^\star)\right]+W_n + \frac{1}{2\gamma}||u_n-y||^2  
&\le& \frac{\tilde C_n(y)}{2\gamma}
%T_nH(y^\star)-W_n\ge T_n \calL(x^e_n,y) + \frac{||u_n-y||^2}{2\gamma} - \frac{\tilde C_n(y)}{2\gamma} 
\end{eqnarray}

\end{proof}

Our next goal is to upper bound quantities $||u_k-y^\star||$. 
We use the same argument as in~\cite{Schmidt:11,AujolDossal:15}. It relies on the following fact.
\begin{lemma}{\cite[Lemma 1]{Schmidt:11}}\label{lemma:Schmidt:seq}
Assume that a nonnegative sequence $\{a_n\}$ satisfies $a_n^2\le S_n+\sum_{k=1}^n\lambda_ka_k$
where $\{S_k\}$ is a nondecreasing sequence, $S_0\ge a_0^2$ and $\lambda_k\ge 0$ for all $k$. Then
$$
a_n\le A_n + \sqrt{S_n+A_n^2}, \qquad A_n=\frac 12 \sum_{k=1}^n \lambda_k 
$$
\end{lemma}
Plugging $y=y^\star$ into~\eqref{eq:JHBASJBFAKSFAS} and observing that $\calL(x^e_n,y^\star)-H(y^\star)\ge 0$
 gives
$$
||u_n-y^\star||^2\le \tilde C_n(y^\star)
$$
Therefore, the sequence $\{a_n\}$ defined via $a_n=||u_n-y_\star||$ satisfies the precondition of Lemma~\ref{lemma:Schmidt:seq} with
$$
B_n=\sum_{k=1}^n\gamma t_k^2\varepsilon_k
\qquad
S_n=||y_0-y^\star||^2+2B_n
\qquad
\lambda_k =2t_k\sqrt{2\gamma\varepsilon_k}
\qquad
A_n=\sum_{k=1}^nt_k\sqrt{2\gamma\varepsilon_k}
$$
We can thus conclude
$$
||u_n-y^\star||\le A_n+\sqrt{||y_0-y^\star||^2+2B_n+A_n^2}\le ||y_0-y^\star|| + 2A_n+\sqrt{2B_n}
$$
By observing that $||u_k-y||\le ||u_k-y^\star|| + ||y-y^\star||$ we obtain
\begin{eqnarray*}
||u_k-y||
&\le& ||y-y^\star||+||y_0-y^\star|| + 2A_k+\sqrt{2B_k} \\
&\le& ||y-y^\star||+||y_0-y^\star|| + 2A_n+\sqrt{2B_n} \qquad\quad\forall 1\le k\le n
\end{eqnarray*}
Plugging this bound into~\eqref{eq:OAISGA} gives the desired result:
\begin{eqnarray*}
 \tilde C_n(y)&\le& 
||y_0-y||^2+2B_n+2A_n\cdot \left(
||y-y^\star||+||y_0-y^\star|| + 2A_n+\sqrt{2B_n}
\right)
\;\;=\;\; C_n(y)
\end{eqnarray*}

%%%%%%%%%%%%%%%%%%%%%%%%%%%%%%%%%%%%%%%%%%%%%%%%%%%%%%%%%%%%%%%%%%%%%%%%%%%%%%%%%%%%%%%%%%%%%%%%%%%%%%%%%%%%%%%%

\section{Algorithm~\ref{alg:FISTA} for solving $\min_x \{f_\calP(x)\:|\:Ax=b\}$: proof of Theorem~\ref{th:LinearConstraint}}

\subsection{Proof of Theorem~\ref{th:LinearConstraint}(a)}
\myparagraph{(a)} Let us fix $n\ge 1$. Using Theorem~\ref{th:PPA} and the strong duality assumption ($H(y^\star)=f_\calP(x^\star)$), we obtain
\begin{equation} \label{eq:GNALASF}
  f_\calP(x^e_n)-f_\calP(x^\star)+\langle y,Ax^e_n-b\rangle
\le \frac{C_n(y)}{2\gamma T_n}
\end{equation}
Plugging $y=0$ gives eq.~\eqref{eq:LinearConstraint:a}. Next, from~\eqref{eq:Cny} we obtain
\begin{align*}
C_n(y)&\le ||y_0||^2+2 ||y_0||\cdot ||y|| + ||y||^2 +2A_n\left(||y|| + ||y^\star||+||y_0-y^\star||+2A_n+\sqrt{2B_n}\right)+2B_n \\
&= ||y||^2 + P\cdot ||y||+Q
\end{align*}
where we denoted
$$
P=2(||y_0||+A_n)\qquad
Q=||y_0||^2+2A_n\left( ||y^\star||+||y_0-y^\star||+2A_n+\sqrt{2B_n}\right)+2B_n
$$
Now denote $r=Ax^e_n-b$, $\Delta=\max\{f_\calP(x^\star)-f_\calP(x^e_n),0\}$, and for scalar $z> 0$ define vector $y=z\cdot \frac{r}{||r||}$
(assuming that $||r||\ne 0$).
From eq.~\eqref{eq:GNALASF} we get
$$
z\cdot ||r||
=\langle y, r\rangle 
\le \Delta + \frac{||y||^2 + P\cdot ||y||+Q}{2\gamma T_n}
= \Delta + \frac{z^2 + P\cdot z+Q}{2\gamma T_n}
$$
and therefore
$$
||r||\le \frac{z}{2\gamma T_n} + \left(\Delta + \frac{Q}{2\gamma T_n}\right)\frac{1}{z} + \frac{P}{2\gamma T_n}
$$
Since $\min_{z>0} \left[\frac z\alpha + \frac \beta z\right]=2\sqrt{\frac{\beta}{\alpha}}$ for $\alpha,\beta>0$, we obtain the claim in eq.~\eqref{eq:LinearConstraint:b}:
$$
||r||
\le 2\sqrt{\frac{\Delta + \frac{Q}{2\gamma T_n}}{2\gamma T_n}} + \frac{P}{2\gamma T_n}
\le \sqrt{\frac{2\Delta }{\gamma T_n}} + \frac{\sqrt{Q}}{\gamma T_n}   + \frac{P}{2\gamma T_n}
= \sqrt{\frac{2\Delta }{\gamma T_n}} + \frac{P/2+\sqrt{Q}}{\gamma T_n}
$$

\subsection{Proof of Theorem~\ref{th:LinearConstraint}(b)}
In this section for a point $x\in\mathbb R^d$ and convex set $C\subseteq\mathbb R^d$ we denote $x_C$ to be the projection of $x$ to $C$, and ${\tt dist}(x,C)=||x-x_C||$.
We will prove the following key result.
\begin{theorem}\label{th:polytopeAngle}
Consider polytope $\calP\subseteq\mathbb R^d$ and linear subspace ${\cal{H}}\subseteq{\mathbb R}^d$ with $\calQ=\calP\cap \cal{H}\ne \varnothing$.
There exists constant $\beta$ such that
\begin{equation}\label{eq:GALSKAS}
{\tt dist}(x,\calQ) \le \beta \cdot {\tt dist}(x,\calH)\qquad\forall x\in\calP
%{\tt dist}(x,\calP\cap{\cal H})\le \tilde\beta\cdot {\tt dist}(x,{\cal H})\qquad \forall x\in\calP
\end{equation}
\end{theorem}
This would easily imply Theorem~\ref{th:LinearConstraint}(b). Indeed,
let ${\cal H}=\{x\in\mathbb R^d\::\:Ax-b=0\}$, then for any
$x\in\calP$ we have $||x-x_{\cal H}||\le
\frac{||Ax-b||}{\sigma_{\min}(A)}$, where $\sigma_{\min}(A)$ denotes
the smallest singular value of $A$. This easily follows from the
projection formula $x_{\cal H} = x - A^*(AA^*)^{-1}(Ax-b)$ and the
fact that any $A$ can be factorized as $A=U\Sigma V^*$ so that
$A^*(AA^*)^{-1} = V\Sigma^{-1}U^*$ and $\norm{V\Sigma^{-1}U^*} =
\sigma^{-1}_{\min}(A)$. It follows from
Theorem~\ref{th:polytopeAngle} that $||\tilde x-x||\le
\frac{\beta}{\sigma_{\min}(A)}||Ax-b||$ where $\tilde
x=x_{\calP\cap{\cal H}}$.  Since $\tilde x\in\calP\cap\cal H$, we have
$f_\calP(\tilde x)\ge f_\calP(x^\star)$.  Convexity of $f_\calP$ gives
$$
f_\calP(x)
\ge f_\calP(\tilde x)+\langle \nabla f(\tilde x),x-\tilde x\rangle
\ge f_\calP(x^\star)-||\nabla f(\tilde x)||\cdot ||\tilde x-x|| 
 \ge f_\calP(x^\star)-||\nabla f(\tilde x)||\cdot \frac{\beta}{\sigma_{\min}(A)}||Ax-b||
$$
which yields Theorem~\ref{th:LinearConstraint}(b), since 
 $||\nabla f(\tilde x)||$ is bounded on $\calP$.

In the remainder of this section we prove Theorem~\ref{th:polytopeAngle}. 
We need to show that $\sup_{x\in\calP-\calH}\beta(x)<+\infty$ where we defined
$\beta(x)=\frac{{\tt dist}(x,\calQ)}{{\tt dist}(x,\calH)}$.
(Note that ${\tt dist}(x,\calH)>0$ for $x\in\calP-\calH$, and ${\tt dist}(x,\calH)={\tt dist}(x,\calQ)=0$ for $x\in\calP\cap\calH=\calQ$).

Clearly,  set $\calQ=\calP\cap \calH$ is a polytope.
Let $\mathbb F_\calP,\mathbb F_\calQ$ be the sets of faces of $\calP,\calQ$ respectively.
For a pair of faces $(\calA,\calB)\in\mathbb F_\calP\times \mathbb F_\calQ$ define 
\begin{eqnarray}
{\calB^+_\circ} &=& \{ x\in\calP \:|\: \mbox{$\calB$ is the minimal face of $\calQ$ containing $x_\calQ$} \} \\
{\calB^+} &=& \{x\in\calP\:|\:x_\calQ=x_\calB=x_{{\tt affine}(\calB)} \in\calB\} \\
\beta(\calA,\calB) &=& \sup_{x\in(\calA\cap\calB^+)-\calH} \beta(x) \label{eq:betaAB} 
%\beta(\calA,\calB)&=&\inf\;\{\;\beta\in[0,+\infty]\:\::\:\:||x-x_\calQ|| \le \beta \cdot ||x-x_{{\cal H}}||\qquad\forall x\in[\calA,\calB]\;\} \label{eq:betaAB}
\end{eqnarray}
where ${\tt affine}(\calB)$ is the affine hull of $\calB$, and we assume that $\sup\varnothing=-\infty$.
Denote 
$$
\mathbb F_{\calP\calQ}=\{(\calA,\calB)\in\mathbb F_\calP\times\mathbb F_\calQ\:|\:(\calA\cap\calB^+)-\calH\ne\varnothing\}
$$
We also define $\dim(\calA,\calB)=\dim \calA+\dim \calB$.
\begin{lemma}\label{lemma:GLAJSFALSGHASFH}
For each $\calB\in\mathbb F_{\calQ}$  set $\calB^+$ are compact, and $\calB^+_\circ\subseteq \calB^+$.
\end{lemma}
\begin{proof}
The first claim follows from continuity of projections and compactness of faces.
To show the second claim, consider $x\in\calB^+_\circ$. 
Clearly, $x_\calB=x_\calQ$.
Also, $x_\calB$ belongs to the interior of $\calB$ (and thus $x_{{\tt affine}(\calB)}=x_\calB$), otherwise  there would exist a face $\calB'\subsetneq\calB$ of $\calQ$ containing
$x_\calB$, contradicting minimality of $\calB$.
\end{proof}

From this lemma we can conclude that $\calP-\calH=\bigcup_{(\calA,\calB)\in\mathbb F_{\calP\calQ}}((\calA\cap\calB^+)-\calH)$.
(Indeed, for each $x\in\calP-\calH$ pick arbitrary face $\calA\in\mathbb F_\calP$ containing $x$, and minimal face $\calB\in\mathbb F_\calQ$ containing $x_Q$.
Then $x\in \calA$, $x\in \calB_\circ^+\subseteq \calB^+$ and so $x\in(\calA\cap\calB^+)-\calH$, implying that $(\calA,\calB)\in\mathbb F_{\calP\calQ}$).
It now suffices to show that
$\beta(\calA,\calB)<+\infty$ for all $(\calA,\calB)\in\mathbb F_{\calP\calQ}$.
Pair $(\calA,\calB)\in\mathbb F_{\calP\calQ}$ 
 will be called {\em $\calQ$-nonadjacent} if $\calA\cap\calB^+\cap \calB\cap\calH$ is empty, 
 and {\em $\calQ$-adjacent} otherwise. It is easy to show the boundedness of $\beta(\calA,\calB)$
 for pairs of the former type.

\begin{lemma}\label{lemma:nonadjacent}
If $(\calA,\calB)$ is $\calQ$-nonadjacent then $\beta(\calA,\calB)<+\infty$. % then $(\calA,\calB)$ is $Q$-adjacent.
%The set $\calA\cap\calB^+\cap \calB\cap\calH$ is non-empty. 
\end{lemma}
\begin{proof}
We will prove the contrapositive statement.
Suppose that $\beta(\calA,\calB)=+\infty$, then set $\calA\cap\calB^+$ contains a sequence $x^1,x^2,\ldots$ with $\lim_{i\rightarrow \infty}\beta(x^i)=+\infty$.
Since values ${\tt dist}(x^i,\calQ)$ are upper-bounded, we must have $\lim_{i\rightarrow \infty}{\tt dist}(x^i,\calH)=0$.
Let $y$ be a limit point of sequence $\{x^i\}$, then ${\tt dist}(y,\calH)=0$
and  $y\in \calA\cap\calB^+$ by compactness of $\calA$ and $\calB^+$.

We have $y\in\calQ=\calP\cap\calH$, and so $y_\calQ=y$.
Since $y\in\calB^+$, we also have $y_\calQ\in\calB$, and so $y\in\calB$.
We showed that $y\in\calA\cap\calB^+\cap \calB\cap\calH$, and so $(\calA,\calB)$ is $\calQ$-adjacent.
\end{proof}

Next, we analyze $\calQ$-adjacent pairs.

\begin{lemma}\label{lemma:ProjectionIsLinear}
Consider a $\calQ$-adjacent pair $(\calA,\calB)$ and points $x\in(\calA\cap\calB^+)-\calH$ and $y\in\calA\cap\calB^+\cap \calB\cap\calH$.
For a value $\lambda\ge 0$ define $x{(\lambda)}=\lambda x + (1-\lambda) y$. \\
(a)
For every linear subspace $\calI$ containing $y$
there exists constant $\tau>0$ (that depends on $x,y,\calI$) such that ${\tt dist}(x(\lambda),\calI)=\tau\lambda$ for all $\lambda\ge 0$. \\
(b) If $\lambda>0$ and $x(\lambda)\in\calA\cap\calB^+$ then $x(\lambda)\notin\calH$ and $\beta(x(\lambda))=\beta(x)$.
%and points $x\in\mathbb R^d$, $y\in\calI$.
%For a value $\lambda\ge 0$ define $x{(\lambda)}=\lambda x + (1-\lambda) y$.
%Then there exists constant $\tau>0$ such that ${\tt dist}(x(\lambda),\calI)=\tau\lambda$.
\end{lemma}
\begin{proof}
\myparagraph{(a)}
By applying translations and orthogonal rotations we can assume w.l.o.g.\ that $\calI=\{z\in\mathbb R^d\:|\:z_1=z_2=\ldots=z_m=0\}$ for some $m\le d$,
and $y=(0,0,\ldots,0)$. (Such transformations do not affect the claim of the lemma).
We have $x(\lambda)=\lambda\cdot x$ and ${\tt dist}(x(\lambda),\calI)=\lambda\cdot(x_1^2+x_2^2+\ldots+x_m^2)^{1/2}$ for any $\lambda\ge 0$.
This implies the claim.

\myparagraph{(b)}
Let $\tau_1$ and $\tau_2$ be the values from part (a) applied respectively to linear subspaces $\calI={\tt affine}(\calB)$ and $\calI=\calH$.
For any $\lambda>0$ with $x(\lambda)\in\calA\cap\calB^+$ (and in particular for $\lambda=1$)
we have ${\tt dist}(x(\lambda),\calQ)={\tt dist}(x(\lambda),{\tt affine}(\calB))=\tau_1 \lambda$
and ${\tt dist}(x(\lambda),\calH)=\tau_2 \lambda$. The claim follows.
\end{proof}

\begin{lemma}\label{lemma:LASJFALSF}
For every $\calQ$-adjacent pair $(\calA,\calB)$ there exists pair $(\calA',\calB')\in\mathbb F_{\calP\calQ}$
with $\beta(\calA',\calB')\ge \beta(\calA,\calB)$ and $\dim(\calA',\calB')<\dim(\calA,\calB)$.
\end{lemma}
\begin{proof}
Consider $x\in(\calA\cap\calB^+)-\calH$. Using Lemma~\ref{lemma:ProjectionIsLinear}(b),
we conclude that there exists $\hat x\in(\calA\cap\calB^+)-\calH$ such that $\beta(\hat x)=\beta(x)$
and $\hat x$ lies on the boundary of $\calA\cap\calB^+$.
(To obtain such $\hat x$, we can take maximum value of $\hat \lambda$ with $x(\hat \lambda)\in\calA\cap\calB^+$, and set $\hat x=x(\hat \lambda)$;
the maximum exists since set $\calA\cap\calB^+$ is compact and $y\ne x$.)

It now suffices to show that $\hat x\in\calA'\cap(\calB')^+$ for some $(\calA',\calB')\in\mathbb F_\calP\times\mathbb F_\calQ$ 
with $\dim(\calA',\calB')<\dim(\calA,\calB)$. (This will imply that $(\calA',\calB')\in\mathbb F_{\calP\calQ}$ since $\hat x\notin\calH$).
Two cases are possible:
\begin{itemize}
\item $\hat x$ lies on the boundary of $\calA$. Then there exists face $\calA'\subsetneq\calA$ of $\calP$ containing 
 is the minimal face of $\calP$ containing $\hat x$. We then have $\hat x\in\calA'\cap\calB^+$, as desired.
\item $\hat x$ lies on the boundary of $\calB^+$. This means that $\hat x_\calQ=\hat x_\calB\in\calB$ lies on boundary of $\calB$.
Let $\calB'$ be the minimal face of $\calQ$ containing $\hat x_\calQ$; by the previous fact we must have $\calB'\subsetneq\calB$.
By definition, $\hat x\in(\calB')^+_\circ$, and so $\hat x\in(\calB')^+$ by Lemma~\ref{lemma:GLAJSFALSGHASFH}.
We then have $\hat x\in\calA\cap(\calB')^+$, as desired.
\end{itemize}
\end{proof}

By applying Lemma~\ref{lemma:LASJFALSF} repeatedly (with an induction on $\dim(\calA,\calB)$) we conclude that
for every $\calQ$-adjacent pair $(\calA,\calB)$ there exists $\calQ$-nonadjacent pair $(\calA',\calB')\in\mathbb F_{\calP\calQ}$
with $\beta(\calA',\calB')\ge \beta(\calA,\calB)$. Combining this fact with Lemma~\ref{lemma:nonadjacent} concludes the proof.

\subsection{Proof of Theorem~\ref{th:LinearConstraint}(c)}
Denote $r_n= ||Ax^e_n-b||$ and $\Delta_n=\max\{f_\calP(x^\star)-f_\calP(x^e_n),0\}$.
Since sequences $\{A_n\}$ and $\{B_n\}$ are bounded and $T_n=\Theta(n^2)$, by Theorem~\ref{th:LinearConstraint}(a)
there exist positive constants $P,Q$ such that 
 $r_n\le P\frac{\sqrt{\Delta_n}}{n} + \frac{Q}{n^2}$ for sufficiently large $n$'s,
 or equivalently $n^2r_n\le Pn\sqrt{\Delta_n}+Q$.
Also, we have $\Delta_n\le \beta r_n$ by Theorem~\ref{th:LinearConstraint}(b),
and so $n^2r_n\le Pn\sqrt{\beta r_n}+Q$.
Denote $z_n=n\sqrt{r_n}$, then
$z_n^2-P\sqrt{\beta}z_n-Q\le 0$.
This means that $z_n\le z^\ast$ where $z^\ast$ is the positive solution of quadratic equation $z^2-P\sqrt{\beta}z-Q= 0$.
We showed that $n\sqrt{r_n}=O(1)$, or equivalently $r_n=O(1/n^2)$.

%%%%%%%%%%%%%%%%%%%%%%%%%%%%%%%%%%%%%%%%%%%%%%%%%%%%%%%%%%%%%%%%%%%%%%%%%%%%%%%%%%%%%%%%%%%%%%%%%%%
%%%%%%%%%%%%%%%%%%%%%%%%%%%%%%%%%%%%%%%%%%%%%%%%%%%%%%%%%%%%%%%%%%%%%%%%%%%%%%%%%%%%%%%%%%%%%%%%%%%
%%%%%%%%%%%%%%%%%%%%%%%%%%%%%%%%%%%%%%%%%%%%%%%%%%%%%%%%%%%%%%%%%%%%%%%%%%%%%%%%%%%%%%%%%%%%%%%%%%%
%%%%%%%%%%%%%%%%%%%%%%%%%%%%%%%%%%%%%%%%%%%%%%%%%%%%%%%%%%%%%%%%%%%%%%%%%%%%%%%%%%%%%%%%%%%%%%%%%%%
%%%%%%%%%%%%%%%%%%%%%%%%%%%%%%%%%%%%%%%%%%%%%%%%%%%%%%%%%%%%%%%%%%%%%%%%%%%%%%%%%%%%%%%%%%%%%%%%%%%
%%%%%%%%%%%%%%%%%%%%%%%%%%%%%%%%%%%%%%%%%%%%%%%%%%%%%%%%%%%%%%%%%%%%%%%%%%%%%%%%%%%%%%%%%%%%%%%%%%%
%%%%%%%%%%%%%%%%%%%%%%%%%%%%%%%%%%%%%%%%%%%%%%%%%%%%%%%%%%%%%%%%%%%%%%%%%%%%%%%%%%%%%%%%%%%%%%%%%%%

\section{Properties of the Frank-Wolfe gap: proofs of Lemmas~\ref{lemma:FWgap1} and~\ref{lemma:FWgap2}}

\subsection{Proof of Lemma~\ref{lemma:FWgap1}}\label{sec:lemma:FWgap1:proof}
Let $x^\star\in\argmin_{x\in P}g(x)$. Since the gradient of $g$ is $L$-Lipschitz, for any $d$ with $\hat x+d\in \calP$ we have 
$$
g(x^\star)\le g(\hat x + d)\le g(\hat x)+\langle \nabla g(\hat x),d \rangle + \tfrac{L}2||d||^2
$$
Let us plug $d=\gamma (s-\hat x)$ where $\gamma\in[0,1]$ and  $s=\lmo_\calP(\nabla g(\hat x))$. We obtain
$$
-c=g(x^\star)-g(\hat x)
\le \gamma \langle\nabla g(\hat x),s-\hat x\rangle + \frac{\gamma^2}{2} L ||s-x||^2
\le - b \gamma + \frac{a\gamma^2}{2}  
%\le \gamma {\tt gap}^{\tt FW}(\hat x;g_\calP) + \frac{\gamma^2}{2} L D^2
$$
where we denoted $c=g_\calP^\downarrow(\hat x)$, $b=\langle\nabla g(\hat x),\hat x-s\rangle={\tt gap}^{\tt FW}(\hat x;g_\calP)$ and $a=LD^2\ge L ||s-x||^2$.
Optimizing the RHS in the last inequality over $\gamma$ gives
%For each $a>0$ and $b\ge 0$ we have
\begin{equation}\label{eq:FALSF}
c\ge
\max_{\gamma\in[0,1]} \left( -\frac{a\gamma^2}{2}  + b \gamma \right) = 
\begin{cases}
b-\frac{a}{2} & \mbox{if } b\ge a \quad\quad (\gamma=1) \\
\frac{b^2}{2a} & \mbox{if } b\le a \quad\quad (\gamma=\frac{b}{a}) 
\end{cases}
%\qquad\qquad\quad\forall a>0,b\ge 0
\end{equation}
To prove the second inequality in Lemma~\ref{lemma:FWgap1},
we need to show two claims:
\begin{itemize}
\item[(i)] If $c>a/2$ then $b\le c+a/2$. Suppose this is false: $c>a/2$ and $c<b-a/2$.
From~\eqref{eq:FALSF} we get $b<a$, and so $c<b-a/2<a-a/2=a/2$ - a contradiction.
\item[(ii)] If $c\le a/2$ then $b\le \sqrt{2ac}$. Suppose this is false: $c\le a/2$ and $c<b^2/2a$.
From~\eqref{eq:FALSF} we get $b>a$ and thus $c\ge b-a/2$, implying $c>a-a/2=a/2$ - a contradiction.
\end{itemize}

\subsection{Proof of Lemma~\ref{lemma:FWgap2}}

Let $\tilde \calL_{\gamma,\bar y}(x,y)=\langle Kx,
y\rangle+f(x)-h^\ast(y)-\frac{1}{2\gamma}||y-\bar y||^2$ be the
function that includes all terms of $\calL_{\gamma,\bar y}$ except for
$\delta_P(x)$, so that $\calL(x,y)=\tilde\calL(x,y)+\delta_P(x)$.
Denote $g(x)=\max_y \tilde \calL_{\gamma,\bar y}(x,y)$, $\hat
g(x)=\tilde \calL_{\gamma,\bar y}(x,\hat y)$ and $\hat
g_\calP(x)=g(x)+\delta_\calP(x)=\calL_{\gamma,\bar y}(x,\hat y)$.

By construction, we have $\calF_{\gamma,\bar y}(x)=g(x)+\delta_P(x)$
and $\hat g(x)\le g(x)$ for any $x$.  Furthermore, functions $g$ and
$\hat g$ are differentiable convex with a Lipschitz continuous gradient (since
the same holds for function $f(x)$ and also for function $\max_y[\langle Kx,y\rangle - h^\ast(y)-\frac{1}{2\gamma}||y-\bar y||^2]$ by Lemma~\ref{lemma:h-gamma}).

Since $\hat y=\argmax_y \tilde \calL_{\gamma,\bar y}(\hat x,y)$, we have $\hat g(\hat x)=\tilde\calL_{\gamma,\bar y}(\hat x,\hat y)=\max_y \tilde\calL_{\gamma,\bar y}(\hat x,\hat y)= g(\hat x)$.
Since $\hat g(\hat x)=g(\hat x)$ and $\hat g(x)\le g(x)$ for all $x\in\calX$, we obtain that $\nabla \hat g(\hat x)=\nabla g(\hat x)$.

Conditions $\hat g(\hat x)=g(\hat x)$ and $\nabla \hat g(\hat x)=\nabla g(\hat x)$
 mean that $\varepsilon={\tt
  gap}^{\tt FW}(\hat x;\calF_{\gamma,\bar y})={\tt gap}^{\tt FW}(\hat
x;\hat g_\calP)$.  Let us apply Lemma~\ref{lemma:FWgap1} to function
$\hat g_\calP(x)$.  We obtain $\hat g_\calP^\downarrow(\hat x)\le
\varepsilon$, or equivalently $\calL_{\gamma,\bar y}(\hat x,\hat y)\le
\varepsilon + \min_x \calL_{\gamma,\bar y}(x,\hat
y)=\varepsilon+H_{\gamma,\bar y}(\hat y)$.

%%%%%%%%%%%%%%%%%%%%%%%%%%%%%%%%%%%%%%%%%%%%%%%%%%%%%%%%%%%%%%%%%%%%%%%%%%%%%%%%%%%%%%%%%%%%%%%%%%%
%%%%%%%%%%%%%%%%%%%%%%%%%%%%%%%%%%%%%%%%%%%%%%%%%%%%%%%%%%%%%%%%%%%%%%%%%%%%%%%%%%%%%%%%%%%%%%%%%%%
%%%%%%%%%%%%%%%%%%%%%%%%%%%%%%%%%%%%%%%%%%%%%%%%%%%%%%%%%%%%%%%%%%%%%%%%%%%%%%%%%%%%%%%%%%%%%%%%%%%
%%%%%%%%%%%%%%%%%%%%%%%%%%%%%%%%%%%%%%%%%%%%%%%%%%%%%%%%%%%%%%%%%%%%%%%%%%%%%%%%%%%%%%%%%%%%%%%%%%%
%%%%%%%%%%%%%%%%%%%%%%%%%%%%%%%%%%%%%%%%%%%%%%%%%%%%%%%%%%%%%%%%%%%%%%%%%%%%%%%%%%%%%%%%%%%%%%%%%%%
%%%%%%%%%%%%%%%%%%%%%%%%%%%%%%%%%%%%%%%%%%%%%%%%%%%%%%%%%%%%%%%%%%%%%%%%%%%%%%%%%%%%%%%%%%%%%%%%%%%
%%%%%%%%%%%%%%%%%%%%%%%%%%%%%%%%%%%%%%%%%%%%%%%%%%%%%%%%%%%%%%%%%%%%%%%%%%%%%%%%%%%%%%%%%%%%%%%%%%%

\section{Properties of function $h_{\gamma,\bar y}(Kx)$: proofs of Lemma~\ref{lemma:h-gamma} and~\ref{lemma:constraint}}\label{sec:proof:last}
\subsection{Proof of Lemma~\ref{lemma:h-gamma}}

  We first show the effect on Moreau envelope of a function $f$ when
  adding a linear function. Let $f$ be a convex functions $f(x)$ and
  define $g(x) = f(x) + \scp{x}{b}$ for some $b \in \calX$. From the
  definition of the Moreau envelope and by completing the squares it
  directly follows that
  \begin{align}\label{eq:Moreau-fpluslin}
    m^\mu_{g}(x) = &\inf_{u} f(u) + \scp{u}{b} +
    \frac{1}{2\mu}\norm{x-u}^2\nonumber\\ = & \inf_{u} f(u) +
    \frac{1}{2\mu}\norm{(x-\mu b) - u}^2 + \scp{x}{b} -
    \frac{\mu}{2}\norm{b}^2\nonumber\\ = & m^\mu_f(x - \mu b) + \scp{x}{b} -
    \frac{\mu}{2}\norm{b}^2.
  \end{align}  
  Next, we rewrite $h_{\gamma, \bar y}(Kx)$ as
  \[
  h_{\gamma, \bar y}(Kx) = -\inf_{y\in\calY} - \scp{Kx}{y} + h^\ast(y) +
  \frac{1}{2\gamma}\norm{y-\bar y}^2,
  \]
  which is now in the form of (minus) an infimal convolution of the
  function $h^*(y)$ minus the linear function $\scp{Kx}{y}$. The
  application of~\eqref{eq:Moreau-fpluslin} immediately yields the
  first representation in terms of the Moreau envelope of $h^*$, that
  is
  \[
  h_{\gamma, \bar y}(Kx) = \frac{\gamma}{2}\norm{Kx}^2 + \scp{Kx}{\bar
    y} - m_{h^*}^\gamma(\bar y + \gamma Kx).
  \]
  The second representation in terms of the Moreau envelope of $h$ is
  obtained by applying the formula~\eqref{eq:mor-conj}.
  
  The gradient formula in terms of the proximal map of $h^*$ can be
  deduced from the gradient formula~\eqref{eq:mor-grad} applied to
  $m_{h^*}^\gamma(\bar y + \gamma Kx)$, which yields
  \[
  \nabla_x h_{\gamma,\bar y}(x) = \gamma K^*Kx + K^*\bar y -
  \gamma K^*\frac{1}{\gamma}\left(\bar y + \gamma Kx - \prox_{\gamma h^*}(\bar y +
  \gamma Kx)\right) = K^*\prox_{\gamma h^*}\left(\bar y + \gamma Kx\right),
  \]
  and Moreau's identity~\eqref{eq:mor-id} easily gives the second
  formula in terms of the proximal map of $h$.

  It remains to compute the Lipschitz constant, which by invoking the
  the nonexpansiveness~\eqref{eq:non-exp} of the proximal map shows
  that
  \begin{align*}
  \norm{K^*\prox_{\gamma h^*}\left(\bar y + \gamma Kx\right) -
    K^*\prox_{\gamma h^*}\left(\bar y + \gamma Ky\right)} \leq &\norm{K}
  \norm{\prox_{\gamma h^*}\left(\bar y + \gamma Kx\right) -
    \prox_{\gamma h^*}\left(\bar y + \gamma Ky\right)}\\
  \leq&\norm{K}\norm{(\bar y + \gamma Kx) - (\bar y + \gamma Ky)}\\
  \leq&\gamma\norm{K}^2 = \gamma L_K^2.
  \end{align*}

\subsection{Proof of Lemma~\ref{lemma:constraint}}
  The proximal map to the indicator function of the linear constraint
  $S = \{y \in \calY: Cy=d\}$ amounts to computing for a given point
  $z \in \calY$ the projection
  \[
  \prox_{\gamma h^*}(z) = \argmin_{y\in \calY} \frac12\norm{y-z}^2,
  \quad \text{s.t.} \; Cy=d.
  \]
  It is a well-known fact that such projection can be computed by
  \[
  \prox_{\gamma h^*}(z) = z + (CC^*)^{-1}(d-Cz),
  \]
  where we have assumed that the matrix $C$ has full row rank, that is
  there are no redundant or conflicting constraints.

  With this result in mind and by combining~\eqref{eq:h-mor}
  with~\eqref{eq:mor-env-prox}, we can easily obtain a closed form
  representation of the function $h_{\gamma, \bar y}(Kx)$, that is
  \begin{align*}
  h_{\gamma, \bar y}(Kx) = & \frac{\gamma}{2}\norm{Kx}^2 +
  \scp{Kx}{\bar y} - m_{h^*}^\gamma(\bar y + \gamma Kx)\\ = &
  \frac{\gamma}{2}\norm{Kx}^2 + \scp{Kx}{\bar y} -
  \delta_S(\prox_{\gamma h^*}(\bar y + \gamma Kx)) -
  \frac{1}{2\gamma}\norm{\bar y + \gamma Kx - \prox_{\gamma h^*}(\bar
    y + \gamma Kx)}^2\\ = & \frac{\gamma}{2}\norm{Kx}^2 +
  \scp{Kx}{\bar y} - \frac{1}{2\gamma}\norm{\bar y + \gamma Kx -
    \prox_{\gamma h^*}(\bar y + \gamma Kx)}^2,
  \end{align*}
  where in the last line we have used the fact that
  $\delta_S(\prox_{\gamma h^*}(z))=0$, for any $z \in
  \calY$. Obviously, $h_{\gamma, \bar y}(Kx)$ is a quadratic function
  and its (linear) gradient with respect to $x$ follows from the
  gradient formula~\eqref{eq:h-grad}, that is
  \[
  \nabla_x h_{\gamma, \bar y}(Kx) = K^*\prox_{\gamma h^*}(\bar y +
  \gamma Kx) = K^*\left(\bar y + \gamma Kx + (CC^*)^{-1}(d-C(\bar y +
  \gamma Kx))\right).
  \]

\bibliographystyle{plain}
\bibliography{FW}

\end{document}